\documentclass{amsart}
\usepackage{amsfonts}
\usepackage{color}

\usepackage{amsmath,amssymb,amsthm, epsfig}

\usepackage{hyperref}

\usepackage{amsmath}

\usepackage{amssymb}

\usepackage[mathscr]{eucal}

\def\R{{\mathbb {R}}}
\def\N{{\mathbb {N}}}

\newlength{\hchng}
\newlength{\vchng}
\def \N {\mathbb{N}}
\def \R {\mathbb{R}}

\def \div {\mathrm{div}}

\def \dist {\mathrm{dist}}

\def \Leb {\mathscr{L}^N}

\newcommand{\defeq}{\mathrel{\mathop:}=}
\newcommand{\lip}{{\rm Lip}}

\newcommand{\ud}{\, d}

\newcommand{\ol}{\overline}
\newcommand{\Om}{\Omega}
\newcommand{\I}{\textrm{I}}
\newcommand{\II}{\textrm{II}}

\newcommand{\trm}{\textrm}
\newcommand{\vp}{\varphi}

\newcommand{\abs}[1]{\left|#1\right|}

 \newcommand{\eps}{{\varepsilon}}

\newtheorem{theorem}{Theorem}[section]
\newtheorem{lemma}[theorem]{Lemma}
\newtheorem{proposition}[theorem]{Proposition}
\newtheorem{corollary}[theorem]{Corollary}
\theoremstyle{definition}

\newtheorem{definition}[theorem]{Definition}

\newtheorem{example}[theorem]{Example}
\theoremstyle{remark}
\newtheorem{remark}[theorem]{Remark}
\numberwithin{equation}{section}

\newcommand{\intav}[1]{\mathchoice {\mathop{\vrule width 6pt height 3 pt depth  -2.5pt
\kern -8pt \intop}\nolimits_{\kern -6pt#1}} {\mathop{\vrule width
5pt height 3  pt depth -2.6pt \kern -6pt \intop}\nolimits_{#1}}
{\mathop{\vrule width 5pt height 3 pt depth -2.6pt \kern -6pt
\intop}\nolimits_{#1}} {\mathop{\vrule width 5pt height 3 pt depth
-2.6pt \kern -6pt \intop}\nolimits_{#1}}}

\begin{document}
	
\title[A limiting free boundary problem with gradient constraint]{A limiting free boundary problem with gradient constraint and Tug-of-War games}

\author[P. Blanc, J.V. da Silva and J.D. Rossi]{Pablo Blanc, Jo\~{a}o V\'{i}tor da Silva and Julio D. Rossi }

\address{Departamento de Matem\'atica, FCEyN - Universidad de Buenos Aires,
\hfill\break \indent IMAS - CONICET
\hfill\break \indent Ciudad Universitaria, Pabell\'on I (1428) Av. Cantilo s/n. \hfill\break \indent Buenos Aires, Argentina.}

\email[P. Blanc]{pblanc@dm.uba.ar}

\email[J.V. da Silva]{jdasilva@dm.uba.ar}

\email[J.D. Rossi]{jrossi@dm.uba.ar}
\urladdr{http://mate.dm.uba.ar/~jrossi}

\begin{abstract}
In this manuscript we deal with regularity issues and the asymptotic behaviour (as $p \to \infty$) of solutions for elliptic free boundary problems of $p-$Laplacian type ($2 \leq  p< \infty$):
\begin{equation*}
      -\Delta_p u(x) + \lambda_0(x)\chi_{\{u>0\}}(x) = 0 \quad \mbox{in} \quad \Omega \subset \R^N,
\end{equation*}
with a prescribed Dirichlet boundary data, where $\lambda_0>0$ is a bounded function and $\Omega$ is a regular domain. First, we prove the convergence as $p\to \infty$ of any family of solutions $(u_p)_{p\geq 2}$, as well as we obtain the corresponding limit operator (in non-divergence form) ruling the limit equation,
$$
\left\{
\begin{array}{rcrcl}
  \max\left\{-\Delta_{\infty} u_{\infty}, \,\, -|\nabla u_{\infty}| + \chi_{\{u_{\infty}>0\}}\right\} & = & 0 & \text{in} & \Omega \cap \{u_{\infty} \geq 0\} \\
  u_{\infty} & = & g & \text{on} & \partial \Omega.
\end{array}
\right.
$$
Next, we obtain uniqueness for solutions to this limit problem together with a number of weak geometric and measure theoretical properties as non-degeneracy, uniform positive density, porosity and convergence of the free boundaries.

Finally, we show that any solution to the limit operator is a limit of value functions
for a specific Tug-of-War game.
\end{abstract}
\maketitle


\section{Introduction}\label{Introd}

 In this article we study diffusion processes governed by quasi-linear operators of $p-$Laplacian type with (possibly) a phase transition regime for solutions, i.e., solutions prescribe a PDE in each \textit{a priori} unknown set (of positivity and negativity respectively)
 $$
    -\Delta_p u + f(u_{-}, u_{+})  = 0 \quad \text{in} \quad \Omega,
 $$
 for a suitable measurable function $f:[0, \infty)\times [0, \infty) \to \R$ with a discontinuity at origin. These models have become mathematically relevant due to their connections with phenomena in applied sciences, as well as several free boundary problems as obstacle type
problems, minimization problems with free boundaries and dead core problems just to mention a few. The problem we are particularly interested is given by
\begin{equation}\label{Eqp-Lapla}
\left\{
\begin{array}{rclcl}
     -\Delta_p u_p(x)+\lambda_0(x)\chi_{\{u_p>0\}}(x) & = & 0 & \mbox{in} & \Omega \\
     u_p(x) & = & g(x) & \mbox{on} & \partial \Omega,
\end{array}
\right.
\end{equation}
where $\Delta_p u = \div(|\nabla u|^{p-2}\nabla u)$ stands for the $p$-Laplace operator, $\lambda_0>0$ is a function (bounded away from zero and from infinity), $g$ is a continuous boundary data and $\Omega \subset \R^N$ is a bounded and regular domain. In this context, $\partial \{u>0\} \cap \Omega$ is the \emph{free boundary} of the problem.

It is worth mentioning that the unique weak solution (cf. \cite[Theorem 1.1 ]{Diaz}) to \eqref{Eqp-Lapla} appears when we minimize the following functional
\begin{equation}\label{eqMin}
    \mathfrak{J}_p[v]  = \int_{\Omega} \left(\frac{1}{p}|\nabla v (x)|^p + \lambda_0(x)v\chi_{\{v>0\}}(x)\right)dx
\end{equation}
over the admissible set $\mathbb{K} =  \left\{v \in W^{1, p}(\Omega)\,\,\,\text{and}\,\, v=g \,\, \text{on} \,\,\partial \Omega\right\}$. Variational problems like \eqref{eqMin} are connected with several applications and were widely studied in the last decades, see
\cite{AP},  \cite{Diaz}, \cite{KKPS}, \cite{LeeShah} and \cite{Manf88}.

In our first result, we infer how weak solutions leave the free boundaries in their positivity set.

\begin{theorem}[{\bf Strong Non-degeneracy}]\label{LGR} Let $u$ be a bounded weak solution to \eqref{Eqp-Lapla}, $\Omega^{\prime} \Subset \Omega$ and let $x_0 \in \overline{\{u >0\}} \cap \Omega^{\prime}$. Then there exists a universal constant $\mathfrak{c}_0$ such that for all $0<r<\min\{1, \dist(\Omega^{\prime}, \partial \Omega)\}$ there holds
\begin{equation}\label{nondeg.est}
   \displaystyle \sup_{\partial B_r(x_0)} \,u(x) \geq \mathfrak{C}_0\left(N, p, \inf_{\Omega} \lambda_0(x)\right)
    r^{\frac{p}{p-1}}.
\end{equation}
\end{theorem}

We also deal with the analysis of asymptotic behaviour as $p$ diverges.
Recently, motivated by game theory (``Tug of-war games''), in \cite{JPR} it is studied the following variational problem
\begin{equation*}
  \left\{
  \begin{array}{rclcl}
    \displaystyle  \Delta_p \,u_p(x)& = & f(x) & \mbox{in}& \Omega \\
    u_p(x) & = & g(x) & \mbox{on} & \partial \Omega
  \end{array}
  \right.
\end{equation*}
with a forcing term $f\geq0$ and a continuous boundary data. In this context, $\{u_p\}_{p\geq 2}$ converges, up to a subsequence, to a limiting function $u_{\infty}$, which fulfils the following problem in the viscosity sense
\begin{equation}\label{eqlimp-Lap2}
  \left\{
  \begin{array}{rclcl}
    \min\left\{ \Delta_{\infty} \, u_{\infty}(x), |\nabla u_{\infty}(x)|- \chi_{\{f>0\}} (x) \right\} & = & 0 & \mbox{in}&  \Omega\\
    u_{\infty}(x) & = & g(x) & \mbox{on} & \partial \Omega,
  \end{array}
  \right.
\end{equation}
where $\Delta_\infty u(x) \defeq \nabla u(x)^TD^2u(x) \cdot \nabla u(x)$
is the \textit{$\infty-$Laplace operator}. (cf. \cite{ACJ} for a survey).
Such limit problems are known as problems with \textit{Gradient constraint}.
Gradient constraint problems like
\begin{equation}\label{eq:gen_grad_constraint}
  \min\{\Delta_\infty u(x), |\nabla u(x)|-h(x) \}=0,
\end{equation}
where $h\ge 0$, appeared in \cite{J}. By considering solutions to
$$
   F_{\varepsilon}[u] \defeq \min\{\Delta_\infty u, |\nabla u|-\epsilon\}=0 \quad
   \text{(resp. of its pair)} \quad F^{\varepsilon}[u] \defeq \max\{\Delta_\infty u, \epsilon-|\nabla u|\}=0
$$
Jensen provides a mechanism to obtain solutions of the infinity-Laplace equation $-\Delta_\infty u=0$ via an approximation procedure. In this context, he proved uniqueness for the infinity-Laplace equation by first showing that it holds for the approximating equations and then sending $\epsilon \to 0$.
A similar strategy was used in the anisotropic counterpart in \cite{ll}, and a variant of \eqref{eq:gen_grad_constraint} appears in the so called $\infty$-eigenvalue problem, see e.g. \cite{jlm}.

We highlight that, in general, the uniqueness of solutions to \eqref{eq:gen_grad_constraint}
is an easy task if $h$ is a continuous function and strictly positive everywhere.
Moreover, uniqueness is known to hold if $h\equiv 0$, see \cite{J}.
Nevertheless, the case $h\ge 0$ yields significant obstacles.
Such a situation resembles the one that holds for the infinity Poisson equation $-\Delta_\infty u=h$,
where the uniqueness is known to hold if $h>0$ or $h\equiv
0$, and the case $h\ge 0$ is an open problem. In this direction, \cite[Theorem 4.1]{JPR} proved uniqueness for \eqref{eq:gen_grad_constraint} in the special case $h=\chi_D$ under the mild topological condition $\overline D=\overline{D^{\circ}}$ on the set $D \subset \R^N$. Furthermore, they show counterexamples where the uniqueness fails if such topological condition is not satisfied, see \cite[Section 4.1]{JPR}. Finally, from a regularity viewpoint, \cite{JPR} also establishes that viscosity solutions to \eqref{eq:gen_grad_constraint} are Lipschitz continuous.

Hence, in our case a natural question arises:
What is the expected behaviour for family of solutions and their free boundaries as $p\to\infty$? 
This question is one of our motivations in order to study existence, uniqueness, regularity and further 
properties for solutions of gradient constraint type models like \eqref{eq:gen_grad_constraint}.

In our next result, we establish existence and regularity of limit solutions. We will assume in this limit procedure that the boundary datum $g$ is a fixed Lipschitz function.

\begin{theorem}[{\bf Limiting problem}]\label{MThmLim1} Let $(u_p)_{p\geq 2}$ be the family of weak solutions to
\eqref{Eqp-Lapla}. Then, up to a subsequence, $u_p \to u_{\infty}$ uniformly in $\overline{\Omega}$.
Furthermore, such a limit fulfils in the viscosity sense
\begin{equation}\label{EqLim}
\left\{
\begin{array}{rcrcl}
  \max\left\{-\Delta_{\infty} u_{\infty}, \,\, -|\nabla u_{\infty}| + \chi_{\{u_{\infty}>0\}}\right\} & = & 0 & \text{in} & \Omega \cap \{u_{\infty} \geq 0\} \\
  u_{\infty} & = & g & \text{on} & \partial \Omega.
\end{array}
\right.
\end{equation}
Finally, $u_{\infty}$ is a Lipschitz continuous function with
$$
    [u_{\infty}]_{\text{Lip}(\overline{\Omega})} \leq \mathfrak{C}(N)\max\left\{1, [g]_{\text{Lip}(\partial \Omega)}\right\}.
$$
\end{theorem}

Notice that \eqref{EqLim} can be written as a fully non-linear second order operator as follows
$$
\begin{array}{rcl}
  F_\infty: \R \times \R^N \times \text{Sym}(N) & \longrightarrow & \R \\
  (s, \xi, X) & \mapsto & \max\left\{-\xi^T X \xi, -|\xi|+ \chi_{\{s>0\}}\right\},
\end{array}
$$
which is non-decreasing in $s$. Moreover, $F_\infty$ is a degenerate elliptic operator in the sense that
$$
   F_\infty(s, \xi, X) \leq F_\infty(s, \xi, Y) \quad \text{whenever} \quad Y\leq X.
$$
Nevertheless, $F_\infty$ is not in the framework of \cite[Theorem 3.3]{CIL}. Then, to prove uniqueness of limit solutions becomes a non-trivial task. We overcome such difficulty by using ideas from \cite[Section 4]{JPR}
and show that solutions to the limit problem are unique.

\begin{theorem}[{\bf Uniqueness}] \label{teo.unicidad}
There is a unique viscosity solution to \eqref{EqLim}. Moreover, a comparison principle holds, i.e.,
if $g_1 \leq g_2$ on $\partial \Omega$ then the corresponding solutions $u_{\infty}^1$
and $u_{\infty}^2$ verify $u_{\infty}^1 \leq u_{\infty}^2$ in $\Omega$.
\end{theorem}

Next, we will turn our attention to the study of several geometric and analytical properties for
limit solutions and their free boundaries. This analysis has been
motivated by the analysis of the asymptotic behaviour of several variational problems (see for example \cite{daSRS1,  JPR, RT, RTU, RW}).
We have a sharp lower control on how limit solutions detach from their free boundaries.

\begin{theorem}[{\bf Linear growth for limit solutions}]\label{MThmLim2} Let $u_{\infty}$ be a uniform limit to solutions $u_p$ of \eqref{Eqp-Lapla} and $\Omega^{\prime} \Subset \Omega$. Then, for any
$x_0 \in \partial \{u_{\infty}>0\} \cap \Omega^{\prime}$ and any $0<r \ll 1$, the following estimate holds:
\begin{equation}\label{estim.22}
  \displaystyle \sup_{B_r(x_0)} u_{\infty}(x) \geq r.
\end{equation}
\end{theorem}

Our main motivation for considering \eqref{EqLim} comes from its connection to modern game theory. Recently, in \cite{PSSW} the authors introduced a two player random turn game called ``Tug-of-war'', and showed that, as the ``step size'' converges to zero, the value functions of this game converge to the unique viscosity solution of the infinity-Laplace equation $-\Delta_\infty u=0$.
We define and study a variant of the Tug-of-war game, that we call {\it Pay or Leave Tug-of-War}, which was inspired by the one in \cite{JPR}.
In our game, one of the players decide to play the usual Tug-of-war or to pass the turn to the other who decides to end the game immediately (and get $0$ as final pay-off) or move and pay $\eps$ (which is the step size).
It is then shown that the value functions of this new game, namely $u^{\varepsilon} $, fulfil a Dynamic Programming Principle (DPP) given by
\[
\begin{split}
u^{\varepsilon}(x) = \min \Bigg\{ \frac12  \left(\sup_{y\in  B_\eps (x)}
u(y) + \inf_{y\in  B_\eps (x)} u^{\varepsilon}(y) \right) ;
\max \Bigg\{0;
 \sup_{y\in  B_\eps (x)} u^{\varepsilon}(y) - \eps \Bigg\}\Bigg\}.
\end{split}
\]
Moreover, we show that the sequence $u^{\varepsilon} $ converge and the corresponding limit is a viscosity solution to \eqref{EqLim}. Therefore, besides its own interest, the game-theoretic scheme provides an alternative mechanism to prove the existence of a viscosity solution to \eqref{EqLim}.

\begin{theorem}
\label{corofin}
Let $u^{\varepsilon}$ be the value functions of the game previously described. Then,
it holds that
\[
\begin{split}
u^{\eps}\to u \qquad
\textrm{uniformly in}\quad\overline{\Omega},
\end{split}
\]
being $u$ the unique viscosity solution to equation~\eqref{EqLim}.
\end{theorem}

Notice that we have been able to obtain a game approximation for a free boundary problem
that involves the set where the solution is positive, $\{u>0\}$. This task involves the
following difficulty, if one tries to play with a rule of the form "one player sells the turn when the expected
payoff is positive", then the value of the game will not be well defined since this rule
is an anticipating strategy (the player needs to see the future in order to decide whet he is going
to play). We overcome this difficulty by letting the other player the chance to stop
the game (and obtain 0 as final payoff in this case) or buy the turn (when the first player gives this
option). In this way we obtain a set of rules that are non-anticipating and give a DPP that
can be seen as a discretization of the limit PDE.

Finally, we include several other quantitative/qualitative properties for free
boundaries of limit solutions, among which we include a result on convergence of the free boundaries
$$
  \partial \{u_p > 0\} \to \partial \{u_{\infty} > 0\}\quad \mbox{as} \quad  p\to \infty,
$$
in the sense of the Hausdorff distance (see Theorem \ref{MThmLim5}).
Moreover, we deal with regularity of certain limit free boundary points under an appropriate geometric condition (see Theorem \ref{RegFB}). At the end we collect some examples in order to illustrate some features of the limit problem.

\section{Preliminaries}\label{Section2}

\begin{definition}[{\bf Weak solution}] $u \in W^{1, p}_{\text{loc}}(\Omega)$ is a weak supersolution (resp. subsolution) to
\begin{equation}\label{eqweaksol}
     -\Delta_p u = \Psi(x, u) \quad \text{in }  \Omega,
\end{equation}
if for all $0\leq \varphi \in C^1_0(\Omega)$ it holds
$$
  \displaystyle \int_{\Omega} |\nabla u|^{p-2}\nabla u\cdot \nabla \varphi(x)\,dx \geq \int_{\Omega} \Psi(x, u)\varphi(x)\,dx \quad \left(\text{resp.}\,\,  \leq \int_{\Omega} \Psi(x, u)\,dx\right).
$$
Finally, $u$ is a weak solution to \eqref{eqweaksol} when it is simultaneously a super-solution and a  sub-solution.
\end{definition}

Since we are assuming that $p$ is large, then \eqref{Eqp-Lapla} is not singular at points where the gradient vanishes. Consequently, the mapping
$$
   x \mapsto \Delta_p \phi(x) = |\nabla \phi(x)|^{p-2}\Delta \phi(x) + (p-2)|\nabla \phi(x)|^{p-4}\Delta_{\infty} \phi(x)
$$
is well-defined and continuous for all $\phi \in C^2(\Omega)$.

Taking into account that the limiting solutions need not be smooth and the fact that the
infinity-Laplace operator is not in divergence form, we must use the appropriate notion of weak
solutions.
Next, we introduce the notion of viscosity solution to \eqref{Eqp-Lapla}. We refer to the survey \cite{CIL} for the general theory of viscosity solutions.

\begin{definition}[{\bf Viscosity solution}]
  An upper (resp. lower) semi-continuous function $u: \Omega \to \R$ is called a viscosity subsolution
  (resp. supersolution) to \eqref{Eqp-Lapla} if, whenever $x_0 \in \Omega$ and $\phi \in C^2(\Omega)$ are such that $u-\phi$ has a strict local maximum (resp. minimum) at $x_0$, then
$$
   -\Delta_p \phi(x_0) + \lambda_0(x_0)\chi_{\{\phi>0\}}(x_0)\leq 0 \quad (\text{resp.} \,\,\,\geq 0).
$$
Finally, $u \in C(\Omega)$ is a viscosity solution to \eqref{Eqp-Lapla} if it is
simultaneously a viscosity subsolution and a viscosity supersolution.
\end{definition}

Analogously we state the definition of viscosity solution to \eqref{EqLim}.

\begin{definition}\label{def:viscosolweak}
An upper semi-continuous (resp. lower semi-continuous) function $u\colon\Om\to\R$ is a
viscosity subsolution (resp. supersolution) to \eqref{EqLim} in $\Om$ if, whenever
$ x_0 \in\Om$ and $\varphi\in C^2(\Om)$ are such that
$u-\varphi$ has a strict local maximum (resp. minimum) at $ x_0$, then
\begin{equation}
\label{subineqs}
\max\{-\Delta_\infty \vp(x), \chi_{\{u\geq 0\}}(x_0)-\abs{\nabla \vp(x_0)}\}\le 0
\end{equation}
respectively
\begin{equation}\label{subineqs2}
\max\{-\Delta_\infty \vp(x), \chi_{\{u> 0\}}(x_0)-\abs{\nabla \vp(x_0)}\}\geq 0.
\end{equation}

Finally, a continuous function $u\colon\Om\to\R$ is a \emph{viscosity
solution} to \eqref{EqLim} in $\Om$ if it is both a viscosity
subsolution and a viscosity supersolution.
\end{definition}

Remark that, since \eqref{subineqs} does not depend on $\phi(x_0)$, we can assume that $\phi$ satisfies $u(x_0) = \phi(x_0)$ and $u(x)<\phi(x)$, when $x \neq x_0$. Analogously, in \eqref{subineqs2} we can assume that $u(x_0) = \phi(x_0)$ and
$u(x)>\phi(x)$, when $x \neq x_0$.
Also we remark that \eqref{subineqs} is equivalent to
$$
   -\Delta_{\infty} \phi(x_0) \leq 0 \quad \text{and} \quad -|\nabla \phi(x_0)|+1.\chi_{\{u>0\}}(x_0) \leq 0;
$$
and that \eqref{subineqs2} is equivalent to
$$
   -\Delta_{\infty} \psi(x_0) \geq 0 \quad \text{or} \quad -|\nabla \psi(x_0)|+1.\chi_{\{u\geq 0\}}(x_0) \geq 0.
$$

The following lemma gives a relation between weak and viscosity sub and super-solutions to \eqref{Eqp-Lapla}.

\begin{lemma}\label{EquivSols} A continuous weak sub-solution (resp. super-solution) $u \in W_{\text{loc}}^{1,p}(\Omega)$ to \eqref{Eqp-Lapla} is a viscosity sub-solution (resp. super-solution) to
$$
  -\left[ |\nabla u(x)|^{p-2} \Delta u(x) + (p-2)|\nabla u(x)|^{p-4}\Delta_{\infty} u (x)
  \right] = -\lambda_0(x)\chi_{\{u>0\}}(x) \quad \text{in} \quad \Omega.
$$
\end{lemma}

\begin{proof} Let us proceed for the case of super-solutions. Fix $x_0 \in \Omega$ and $\phi \in C^2(\Omega)$ such that $\phi$ touches $u$ by bellow, i.e., $u(x_0) = \phi(x_0)$ and $u(x)> \phi(x)$ for $x \neq x_0$. Our goal is to show that
$$
  -\left[|\nabla \phi(x_0)|^{p-2}\Delta \phi(x_0) + (p-2)|\nabla \phi(x_0)|^{p-4}\Delta_{\infty} \phi(x_0)\right] + \lambda_0(x_0)\chi_{\{\phi>0\}}(x_0) \geq 0.
$$
Let us suppose, for sake of contradiction, that the inequality does not hold. Then, by continuity there exists  $r>0$ small enough such that
$$
   -\left[ |\nabla \phi(x)|^{p-2}\Delta \phi(x) + (p-2)|\nabla \phi(x)|^{p-4}\Delta_{\infty} \phi(x)\right] +\lambda_0(x)\chi_{\{\phi>0\}}(x) < 0,
$$
provided that $x \in B_r(x_0)$. Now, we consider
$$
   \Psi(x) \defeq \phi(x)+ \frac{1}{1000}\mathfrak{i}, \quad \text{ where } \quad \mathfrak{i} \defeq \inf_{\partial B_r(x_0)} (u(x)-\phi(x)).
$$
Notice that $\Psi$ verifies $\Psi < u$ on $\partial B_r(x_0)$, $\Psi(x_0)> u(x_0)$ and
\begin{equation}\label{EqPsi}
 -\Delta_p \Psi(x) + \lambda_0(x)\chi_{\{\phi>0\}}(x) <0.
\end{equation}
By extending by zero outside  $B_r(x_0)$, we may use $(\Psi-u)_{+}$ as a test function in \eqref{Eqp-Lapla}. Moreover, since $u$ is a weak super-solution, we obtain
\begin{equation}\label{Eq3.4}
  \displaystyle \int_{\{\Psi>u\}} |\nabla u|^{p-2}\nabla u \cdot \nabla (\Psi-u) dx \geq  -\int_{\{\Psi>u\}} \lambda_0(x)\chi_{\{u>0\}}(x_0)(x)(\Psi-u) dx.
\end{equation}
On the other hand, multiplying \eqref{EqPsi} by $\Psi- u$ and integrating by parts we get
\begin{equation}\label{Eq3.5}
  \displaystyle \int_{\{\Psi>u\}} |\nabla \Psi|^{p-2}\nabla \Psi \cdot \nabla (\Psi-u) dx <  -\int_{\{\psi>u\}} \lambda_0(x)\chi_{\{\phi>0\}}(x)(\Psi-u) dx.
\end{equation}
Next, subtracting \eqref{Eq3.4} from \eqref{Eq3.5} we obtain
$$
\displaystyle  \int\limits_{\{\Psi>u\}} \left(|\nabla \Psi|^{p-2}\nabla \Psi - |\nabla u|^{p-2}\nabla u\right) \cdot \nabla (\Psi-u) dx <  \int\limits_{\{\psi>u\}} \lambda_0(x)\left(\chi_{\{\phi>0\}}(x)-\chi_{\{u>0\}}(x)\right)(\Psi-u)dx <0.
$$
Finally, since the left hand side is bounded by below by $ 2^{-p}\int_{\{\Psi>u\}} |\nabla \Psi- \nabla u|^pdx \geq 0,$ this forces $\Psi \leq u$ in $B_r(x_0)$. However, this contradicts the fact that $\Psi(x_0)>u(x_0)$
and proves the result.

Similarly, one can prove that a continuous weak sub-solution is a viscosity sub-solution.
\end{proof}

\begin{theorem}[{\bf Morrey's inequality}]\label{MorIneq} Let $N<p\leq \infty$. Then for $u \in W^{1, p}(\Omega)$, there exists a constant $\mathfrak{C}(N, p)>0$ such that
$$
  \|u\|_{C^{0, 1-\frac{N}{p}}(\Omega)} \leq \mathfrak{C}(N, p)\|\nabla u\|_{L^{p}(\Omega)}.
$$
\end{theorem}

We must highlight that the dependence of $\mathfrak{C}$ on $p$ does not deteriorate as $p \to \infty$. In fact,
$$
  \mathfrak{C}(N, p) \defeq \frac{2\mathfrak{c}(N)}{|\partial B_1|^{\frac{1}{p}}}\left(\frac{p-1}{p-N}\right)^{\frac{p-1}{p}},
$$
where $\mathfrak{c}(N)>0$ is a dimensional constant.

\section{Non-degeneracy of solutions}\label{SecNond}

This section is devoted to establish a weak geometrical property which plays a key role in the description of how solutions leave their free boundaries. We show \emph{non-degeneracy} of solutions.

\begin{proof}[{\bf Proof of Theorem \ref{LGR}}]
Due to the continuity of solutions, it is enough to prove such a estimate just at points $x_0 \in \{u>0\} \cap \Omega^{\prime}$.
Let us define the scaled function
$$
  u_r(x) \defeq \frac{u(x_0+rx)}{r^{\frac{p}{p-1}}}.
$$
and the auxiliary barrier $$\displaystyle \Psi (x) \defeq \mathfrak{C}_0 |x|^{\frac{p}{p-1}}
\qquad \mbox{with} \quad
  \displaystyle \mathfrak{C}_0 \defeq \frac{p-1}{p}\left(\frac{ \inf_{\Omega}\lambda_0(x)}{ N}\right)^{\frac{1}{p-1}}.
$$
It is easy to check that
$$
     -\Delta_p \Psi + \hat{\lambda}_0\left( x  \right).\chi_{\{\Psi>0\}}(x) \geq 0 \geq -\Delta_p u_r + \hat{\lambda}_0\left( x  \right). \chi_{\{u_r>0\}}(x) \quad \text{in} \quad B_1,
$$
in the weak sense, where $\hat\lambda_0(x) \defeq \lambda_0(x_0 + rx)$.
Now, if $u_r \leq \Psi$ on the whole boundary of $B_1$, then the Comparison Principle yields that
$$
   u_r \leq \Psi \quad \mbox{in} \quad B_1,
$$
which contradicts the assumption that $u_r(0)>0$. Therefore, there exists a point $y \in \partial B_1$ such that
$$
      u_r(y) > \Psi(y) = \mathfrak{C}_0,
$$
The proof finishes by scaling back $u_r$.
\end{proof}

As a by-product of Theorem \ref{LGR} we obtain a sharp growth estimate near the free boundary:

\begin{corollary}[{\bf Sharp growth}]\label{CorNonDeg}
Let $u$ be a non-negative, bounded weak solution to \eqref{Eqp-Lapla} in $\Omega$ and   $\Omega^{\prime} \Subset \Omega$. Given $x_0 \in \{u>0\} \cap \Omega^{\prime}$, there exists a universal constant $\mathfrak{C}_{\sharp}>0$ such that
$$
  u(x_0) \geq \mathfrak{C}_{\sharp}\dist(x_0, \partial \{u>0\})^{\frac{p}{p-1}}.
$$
\end{corollary}
\begin{proof} We argue by contradiction.
Then, there exists a sequence $(x_k)_{k \in \mathbb{N}} \in \{u>0\} \cap \Omega^{\prime}$ with
$$
   d_k \defeq \dist(x_k, \partial \{u>0\} \cap \Omega^{\prime}) \to 0 \quad \text{as} \quad k \to \infty \quad \text{and} \quad u(x_k) \leq  k^{-1} d_k^{\frac{p}{p-1}}.
$$
Now, let us define the auxiliary function $v_k:B_1 \to \R$ by
$$
   v_k(y) \defeq \frac{u(x_k+d_ky)}{d_k^{\frac{p}{p-1}}}.
$$
It is easy to check that
\begin{enumerate}
  \item $-\Delta_p v_k + \lambda_0(x_k+d_ky)\chi_{\{v_k>0\}} = 0$ in $B_1$ in the weak sense.
  \item $v_k(y) \leq \mathfrak{C}(N, p).d_k^{\alpha} + \frac{1}{k} \,\, \forall\,\, y \in B_{\frac{1}{4}}$ due to local H\"{o}lder regularity of weak solutions, see \cite{DB0}.
\end{enumerate}
From the Non-degeneracy Theorem \ref{LGR} and the last sentence we obtain that
\begin{equation}\label{estim.44}
  \displaystyle 0<\mathfrak{C}_0 \left(\frac{1}{4}\right)^{\frac{p}{p-1}} \leq \sup_{B_{\frac{1}{4}}} v_k(y) \leq \mathfrak{C}(N, p)d_k^{\alpha}+\frac{1}{k} \to 0 \quad \text{as} \quad k \to \infty.
\end{equation}
This is a contradiction that concludes the proof.
\end{proof}

\section{The limit problem}

This section is devoted to prove Theorems \ref{MThmLim1} and \ref{MThmLim2} concerning the limit as $p\to \infty$. First, we will prove the existence of a uniform limit for Theorem \ref{MThmLim1} as $p\to\infty$. Remind that since the boundary datum $g$ is assumed to be Lipschitz continuous we can extend it to a Lipschitz function (that we will still call $g$) to the whole $\Omega$.

\begin{lemma}\label{Lemma2.4} Assume $\max\{2, N\}<p < \infty$ and let $u_p \in W^{1, p}(\Omega)$ be a weak solution to \eqref{Eqp-Lapla}. Then,
$$
\|\nabla u_p\|_{L^p(\Omega)} \leq C_1.
$$
Additionally, $u_p \in C^{0, \alpha}(\Omega)$, where $\alpha = 1- \frac{N}{p}$ with the following estimate
  $$
  \frac{|u_p(x)-u_p(y)|}{|x-y|^{\alpha}} \leq C_2,
  $$
where $C_1, C_2>0$ are constants depending on $N$, $ \|\lambda_0\|_{L^{\infty}(\Omega)}$,
$\|g\|_{L^{\infty}(\Omega)}$, $\|\nabla g\|_{L^{\infty}(\Omega)}$.
\end{lemma}

\begin{proof}
The unique weak solution $u_p\in W^{1,p} (\Omega)\cap C(\overline\Omega)$ to  $\Delta_p u_p= \lambda_0
\chi_{\{u_p>0\}}$ with fixed Lipschitz continuous boundary values $g$, can be characterized as being the
minimizer for the functional
$$
    \mathfrak{J}_p[u]   = \int_{\Omega} \frac{|\nabla u|^p}{p} \, dx + \int_{\{u>0\}} \lambda_0 u \, dx
$$
in the set of functions $\mathbb{K} = \{ u \in W^{1,p} (\Omega) \ : \ u =g \text{ on } \partial \Omega\}$.
Using $g$ as test function and the fact that $\|u_p\|_{L^{\infty}(\Omega)}\leq \|g\|_{L^{\infty}(\Omega)}$ we obtain
$$
\begin{array}{rcl}
  \displaystyle \int_{\Omega} |\nabla u_p|^p \, dx & = & \displaystyle
  \int_{\Omega} |\nabla g|^p \, dx + \int_{\{g>0\}} \lambda_0 g \, dx - \int_{\Omega} \lambda_0(u_p)_{+}\, dx \\[10pt]
   & \leq  & \displaystyle C \|\nabla g\|^{p}_{L^\infty (\Omega)} + C\|\lambda_0\|_{L^{\infty}(\Omega)}\|g\|_{L^{\infty}(\Omega)}.
\end{array}
$$
Therefore,
$$
\|\nabla u_p\|_{L^p(\Omega)} \leq C_1.
$$
Next, for $p>N$ by Morrey's estimates we get
$$
   \frac{|u_p(x)-u_p(y)|}{|x-y|^{1- \frac{N}{p}}} \leq \mathfrak{C} \|\nabla u_p\|_{L^p(\Omega)}.
$$
\end{proof}

Next, we show that any family of weak solutions to \eqref{Eqp-Lapla} is pre-compact and therefore we get the existence of a uniform limit (as stated in Theorem \ref{MThmLim1}).

\begin{lemma}[{\bf Existence of limit solutions}]\label{LemExistSol} Let $\{u_p\}_{p>2}$ be a sequence of weak solutions to \eqref{Eqp-Lapla}. Then, there exists a subsequence $p_j \to \infty$ and a limit function $u_{\infty}$ such that
$$
   \displaystyle \lim_{p_j \to \infty} u_{p_j}(x) = u_{\infty}(x)
$$
uniformly in $\Omega$. Moreover, $u_{\infty}$ is Lipschitz continuous with
$$
    [u_{\infty}]_{\text{Lip}(\overline{\Omega})} \leq \limsup_{p_j \to \infty} \mathfrak{C}(N, p_j, \Omega)\|\nabla u_{p_j}\|_{L^{p_j}(\Omega)} \leq \mathfrak{C}(N)\max\left\{1, [g]_{\text{Lip}(\partial \Omega)}\right\}.
$$
\end{lemma}
\begin{proof} Existence of a uniform limit,
$u_{\infty}$, is a direct consequence of our estimates in Lemma \ref{Lemma2.4} using with an Arzel\`{a}-Ascoli compactness criteria. Finally, the last statement holds by passing to the limit in the H\"{o}lder's estimates from Lemma \ref{Lemma2.4}.
\end{proof}

Next, we will show that any uniform limit, $u_\infty$, is a viscosity solution to the limit equation.

\begin{proof}[{\bf Proof of Theorem \ref{MThmLim1}}]
Notice that from the uniform convergence, it holds that $u_{\infty} = g$ on $\partial \Omega$. Next, we  prove that the limit function $u_{\infty}$  is a viscosity solution to
$$
  \max\left\{-\Delta_{\infty} u_{\infty}(x), -|\nabla u_{\infty}(x)|+ \chi_{\{u_{\infty}>0\}}(x)\right\} = 0 \quad \text{in} \quad \Omega.
$$
First, let us prove that $u_{\infty}$ is a viscosity supersolution.
To this end, fix $x_0 \in \{u_{\infty}>0\} \cap \Omega$ and let  $\phi \in C^2(\Omega)$ be a test function such that $u_{\infty}(x_0) = \phi(x_0)$ and the inequality $u_{\infty}(x) > \phi(x)$ holds for all $x \neq x_0$. We want to show that
$$
  - \Delta_{\infty} \phi(x_0) \geq 0 \quad \text{or} \quad -|\nabla \phi(x_0)|+ \chi_{\{u_{\infty}\geq0\}}(x_0) \geq 0.
$$
Notice that if $-|\nabla \phi(x_0)|+ \chi_{\{u_{\infty}\geq0\}}(x_0) \geq 0$ there is nothing to prove. Hence, we may assume that
\begin{equation}\label{eq5.1}
  -|\nabla \phi(x_0)|+\chi_{\{u_{\infty}\geq0\}}(x_0)<0.
\end{equation}
Since, up to a subsequence,  $u_p \to u_{\infty}$ uniformly, there exists a sequence $x_p \to x_0$ such that $x_{p} \to x_0$ such that $u_{p}-\phi$ has a local minimum at $x_{p}$. Since $u_{p}$ is a weak supersolution (and then a viscosity supersolution by Lemma \ref{EquivSols}) to \eqref{Eqp-Lapla} we get
$$
  -\left[|\nabla \phi(x_{p})|^{p-2}\Delta \phi(x_{p}) + (p-2)|\nabla \phi(x_{p})|^{p-4}\Delta_{\infty} \phi(x_{p})\right] \geq -\lambda_0(x_p)\chi_{\{\phi\geq 0\}}(x_p).
$$
Now, dividing both sides by $(p-2)|\nabla \phi(x_{p})|^{p-4}$ (which is not  zero for $p\gg1$ due to \eqref{eq5.1}) we get
$$
  - \Delta_{\infty} \phi(x_{p}) \geq  \frac{|\nabla \phi(x_{p})|^2 \Delta \phi(x_{p})}{p-2} -\left( \frac{\sqrt[p-4]{\lambda_0(x_p)\chi_{\{\phi\geq 0\}}(x_p)}}{|\nabla \phi(x_{p})|}\right)^{p-4}.
$$
Passing the limit as $p \to \infty$ in the above inequality we conclude that
$$
- \Delta_{\infty} \phi(x_0) \geq 0,
$$
which proves that $u_{\infty}$ is a viscosity supersolution.

Now, let us show that $u_{\infty}$ is a viscosity subsolution. To this end, fix $x_0 \in \{u_{\infty}>0\} \cap \Omega$ and a test function $\phi \in C^2(\Omega)$ such that $u_{\infty}(x_0) = \phi(x_0)$ and the inequality $u_{\infty}(x) < \phi(x)$ holds for $x \neq x_0$. We want to prove that
\begin{equation}\label{eq5.2}
  - \Delta_{\infty} \phi(x_0) \leq 0 \quad \text{and} \quad -|\nabla \phi(x_0)|+\chi_{\{u_{\infty}\geq0\}}(x_0) \leq 0.
\end{equation}
One more time, there exists a sequence $x_{p} \to x_0$ such that $u_{p}-\phi$ has a local maximum at $x_{p}$ and since $u_{p}$ is a weak sub-solution (resp. viscosity sub-solution) to \eqref{Eqp-Lapla}, we have that
$$
 - \frac{|\nabla \phi(x_{p})|^2 \Delta \phi(x_{p})}{p-2} - \Delta_{\infty} \phi(x_{p}) \leq  -\left( \frac{\sqrt[p-4]{\lambda_0(x_p)\chi_{\{u_{\infty}\geq0\}}(x_p)}}{|\nabla \phi(x_{p})|}\right)^{p-4} \leq 0.
$$
Thus, letting $p \to \infty$ we obtain $- \Delta_{\infty} \phi(x_0) \leq 0$. Furthermore, if  $-|\nabla \phi(x_0)|+ 1.\chi_{\{u_{\infty}\geq0\}}(x_0) > 0$, as $p \to \infty$, then the right hand side diverges to $-\infty$, giving a contradiction.
Therefore \eqref{eq5.2} holds.

Next, let us establish the limit equation in the null set. To this end, fix $x_0 \in \Omega \cap \{u_{\infty} = 0\}$ and $\phi \in C^2(\Omega)$ such that $u_{\infty}(x_0) = \phi(x_0)=0$ and $u_{\infty}(x) < \phi(x)$ holds for $x \neq x_0$. As before, there exists a sequence $x_{p} \to x_0$ such that $u_{p}-\phi$ has a local minimum at $x_{p}$.
We consider two cases:

\begin{enumerate}
  \item[Case 1:] $u_{p_k}(x_{p_k}) \leq 0$ for a subsequence $(p_k)_{k\geq 1}$.
In this case, since $u_{p_k}$ is a weak super-solution (resp. viscosity super-solution) to \eqref{Eqp-Lapla}, we obtain after passing to the limit as $p_k \to \infty$ that $-\Delta_{\infty} \phi(x_0) \geq 0$.

  \item[Case 2:] $u_{p_k}(x_{p_k}) > 0$ for a subsequence $(p_k)_{k\geq 1}$.
In this case, since $u_{p_k}$ is a weak super-solution (resp. viscosity super-solution) to \eqref{Eqp-Lapla}, we have that
$$
   - \Delta_{p_k} \phi(x_{p_k}) \geq \lambda_0(x_{p_k}).
$$
As in the first part of this proof, we obtain after passing the limit as $p_k \to \infty$ that
$$
  -\Delta_{\infty} \phi(x_0) \geq 0 \quad \text{or} -|\nabla \phi(x_0)| + 1\geq 0
$$
\end{enumerate}

In both cases, we conclude that
$$
  \max\left\{-\Delta_{\infty} \phi(x_0), -|\nabla \phi(x_0)|+ \chi_{\{u\geq 0\}}\right\} \geq 0,
$$
which assures that $u_{\infty}$ is a viscosity super-solution to \eqref{EqLim} in its null set.

Now, fix $x_0 \in \Omega \cap \{u_{\infty} = 0\}$ and $\phi \in C^2(\Omega)$ such that $u_{\infty}(x_0) = \phi(x_0)=0$ and $u_{\infty}(x) > \phi(x)$ holds for $x \neq x_0$. One more time, there exists a sequence $x_{p} \to x_0$ such that $u_{p}-\phi$ has a local maximum at $x_{p}$. As before, let us consider two possibilities:

\begin{enumerate}
  \item[Case 1:] $u_{p_k}(x_{p_k}) \leq 0$ for a subsequence $(p_k)_{k\geq 1}$.
In this case, since $u_{p_k}$ is a weak sub-solution (resp. viscosity sub-solution) to \eqref{Eqp-Lapla}, we obtain
$-\Delta_{\infty} \phi(x_0) \leq 0$. Moreover, we also have $-|\nabla \phi(x_0)|+ \chi_{\{u> 0\}}
= -|\nabla \phi(x_0)| \leq 0$.

  \item[Case 2:] $u_{p_k}(x_{p_k}) > 0$ for an infinity subsequence $(p_k)_{k\geq 1}$.
In this case, since $u_{p_k}$ is a weak sub-solution (resp. viscosity sub-solution) to \eqref{Eqp-Lapla}, we have that
$$
   - \Delta_{p_k} \phi(x_{p_k}) \leq \lambda_0(x_{p_k}).
$$
Once again, we obtain after passing to the limit as $p_k \to \infty$,
$$
  -\Delta_{\infty} \phi(x_0) \leq 0 \quad \text{and} -|\nabla \phi(x_0)|+ 1\leq 0
$$
\end{enumerate}

Therefore, in any of the two cases, we conclude that
$$
  \max\left\{-\Delta_{\infty} \phi(x_0), -|\nabla \phi(x_0)|+ \chi_{\{u> 0\}}\right\} \leq 0,
$$
which shows that $u_{\infty}$ is a viscosity sub-solution to \eqref{EqLim} in its null set.

Finally, to prove that $u_{\infty}$ is $\infty-$harmonic in its negativity set is a standard task and the reasoning is similar to one employed in \cite[Theorem 1]{RT}, \cite[page 384]{RTU} and \cite[Theorem 1.1]{RW}. We omit the details here.
\end{proof}

\begin{proof}[{\bf Proof of Theorem \ref{MThmLim2}}]
 Any sequence of weak solutions $(u_p)_{p\geq 2}$ converge, up to a subsequence, to a limit, $u_{\infty}$,  uniformly in $\Omega$. From Theorem \ref{LGR} we have that
$$
   \displaystyle \sup_{B_r(x_0)} \,u(x) \geq  \mathfrak{C}_{0} r^{\frac{p}{p-1}} \qquad \text{with} \quad  \displaystyle \mathfrak{C}_{0} \defeq \frac{p-1}{p}\left(\frac{ \inf_{\Omega}\lambda_0(x)}{ N}\right)^{\frac{1}{p-1}}.
$$
As before for $\hat{x} \in \overline{\{u_{\infty}>0\}} \cap \Omega^{\prime}$ there exist $x_p \to \hat{x}$ with $x_p \in \overline{\{u_p>0\}} \cap \Omega^{\prime}$. Hence we get,
$$
 \displaystyle \sup_{B_r(x_0)} u_{\infty}(x)  = \lim_{p \to \infty} \sup_{B_r(x_p)} u_p(x) \geq  r.
$$
\end{proof}

As a by-product of previous estimates we prove that any limit solution to \eqref{EqLim} is, near the free boundary,  ``trapped'' between the graph of two multiples of $\dist(\cdot, \partial \{u>0\})$.

\begin{corollary} \label{coro62}
 Let $u_{\infty}$ be a uniform limit to normalized solutions $u_p$ of \eqref{Eqp-Lapla} and $\Omega^{\prime} \Subset \Omega$. Then, for any $x_0 \in \{u_{\infty}>0\} \cap \Omega^{\prime}$ the following estimate holds:
$$
 \mathfrak{C}_1(N)\dist(x_0, \partial \{u>0\}) \leq  u_{\infty}(x_0) \leq \mathfrak{C}_2(N)\dist(x_0, \partial \{u>0\}).
$$
\end{corollary}

\begin{proof}
The upper bound for $u_\infty(x_0)$ is a consequence of Lipschitz regularity. For the remaining inequality, let us suppose, for sake of contradiction that such lower bound does not hold. Then, it exists
a sequence $x_k \in \{u_{\infty}>0\} \cap \Omega^{\prime}$ such that
$$
d_k \defeq \dist(x_k, \partial \{u_{\infty}>0\} \cap \Omega^{\prime}) \to 0 \quad \text{as} \quad k \to \infty \quad \text{and} \quad u(x_k) \leq k^{-1}d_k.
$$
Now, define the auxiliary function $v_k:B_1 \to \R$ by
$$
   v_k(y) \defeq \frac{u_{\infty}(x_k+d_ky)}{d_k}.
$$
From uniform convergence $v_{k, p} \to v_k$ locally uniformly as $k \to \infty$, where
$$
   v_{k, p}(y) \defeq \frac{u_{p}(x_k+d_ky)}{d_k^{\frac{p}{p-1}}}.
$$
Thus,
$$
   -\Delta_p v_{k, p}(y) + \lambda_0(x_k+d_ky)\chi_{\{v_{k, p}>0\}}(y)=0 \quad \text{in} \quad B_1
$$
in the weak sense. From classical regularity estimates $v_{k, p}$ is H\"{o}lder continuous (see Theorem \ref{MorIneq}). After passing to the limit as $k \to \infty$ we infer that $v_k$ is $\alpha-$H\"{o}lder continuous for any $\alpha \in (0, 1)$.
Hence,
$$
  v_k(y) \leq C|y-0|^{\alpha}+v_k(0) \leq \max\{1, C\}\left(d_k^{\alpha} + \frac{1}{k}\right)\,\,\, \forall\,\, y \in B_1.
$$
Finally, from Theorem \ref{MThmLim2}, we obtain
$$
  \displaystyle 0<\frac{1}{2} \leq \frac{1}{d_k}\sup_{B_{\frac{d_k}{2}}(x_k)} u_{\infty}(x) = \sup_{B_{\frac{1}{2}}} v_k(z) \leq \max\{1, C\}\left(d_k^{\alpha} + \frac{1}{k}\right) \to 0 \quad \text{as} \quad k \to \infty.
$$
This contradiction finishes the proof.
\end{proof}

\section{Uniqueness for the limit problem}\label{Uniqu}

Our main goal throughout this section is to show uniqueness of viscosity solutions to
\begin{equation} \label{pppp}
\left\{
  \begin{array}{rclcl}
    \max\left\{ -\Delta_{\infty} \, u_{\infty}(x), \chi_{\{u_{\infty}>0\}}(x) -|\nabla u_{\infty}(x)| \right\} & = & 0 & \mbox{in}&  \Omega\\
    u_{\infty}(x) & = & g(x) & \mbox{on} & \partial \Omega.
  \end{array}
\right.
\end{equation}
Remind that existence of a solution $u_\infty$ was obtained as the uniform limit (along subsequences) of solutions to $p-$Laplacian problems \eqref{Eqp-Lapla}, see Theorem \ref{MThmLim1} for more details. Next, we will deliver the proof of Theorem \ref{teo.unicidad}, which is based on \cite[Section 4]{JPR}. For this reason, we will only include some details.

\begin{proof}[{\bf Proof of Theorem \ref{teo.unicidad}}]
To prove such a result we first construct a function $v$ and then show that any possible viscosity solution to \eqref{pppp} coincides with $v$. To construct such an special $v$ we first consider $\mathfrak{h}$ the unique (see \cite{J}) viscosity solution to
\begin{equation} \label{pppp.h}
\left\{
  \begin{array}{rclcl}
   -\Delta_{\infty} \, \mathfrak{h} (x) & = & 0 & \mbox{in}&  \Omega\\
    \mathfrak{h}(x) & = & g(x) & \mbox{on} & \partial \Omega.
  \end{array}
\right.
\end{equation}
Then, let $z$ be the unique viscosity solution to
\begin{equation} \label{pppp.z}
\left\{
  \begin{array}{rclcl}
    \max\left\{ -\Delta_{\infty} \, z (x), 1 -|\nabla z (x)| \right\} & = & 0 & \mbox{in}&  \Omega\\
    z (x) & = & g(x) & \mbox{on} & \partial \Omega.
  \end{array}
\right.
\end{equation}
Remark that for this problem we have uniqueness, as well as validity of a comparison principle, see
\cite[Theorem 4.5]{JPR}. Hence, we have
$$
z(x) \leq u_\infty (x) \leq \mathfrak{h}(x) \qquad \forall \,\, x \in \Omega.
$$
Moreover, from \cite[Theorem 4.2]{JPR}, we have
$$
z(x) = u_\infty (x) = \mathfrak{h}(x) \qquad \text{in} \quad  \{ x\in \Omega \, : \, \nabla \mathfrak{h} (x) \geq 1 \}.
$$
Now, we modify $z$ in the set $\{x\in \Omega \, : \, z (x) < 0 \}$ to obtain
the function $v$ as follows: let $w$ be the solution to
\begin{equation} \label{pppp.w}
\left\{
  \begin{array}{rclcl}
   -\Delta_{\infty} \, w (x) & = & 0 & \mbox{in}& \{x\in \Omega \, : \, z (x) < 0 \} \\
    w(x) & = & z(x) & \mbox{on} & \partial \{x\in \Omega \, : \, z (x) < 0 \}.
  \end{array}
\right.
\end{equation}
and then we set
\begin{equation}\label{eq.v}
v(x) = \left\{
\begin{array}{ll}
z(x) \qquad & \mbox{for } \{x\in \Omega \, : \, z (x) \geq  0 \}, \\
w(x) \qquad & \mbox{for } \{x\in \Omega \, : \, z (x) <  0 \}.
\end{array}
\right.
\end{equation}
Remark that this function $v$ is uniquely determined by the boundary datum $g$ since all
the involved PDE problems have uniqueness. Moreover, since we have a comparison
principle for the involved PDE problems then we have a comparison principle for $v$, that is,
if $g_1 \leq g_2$ on $\partial \Omega$, then the corresponding functions $v_1$ and $v_2$ verify
$$
v_1 (x) \leq v_2 (x) , \qquad \mbox{in } \Omega.
$$

Now our aim is to show that
$$
   u_\infty = v, \qquad \mbox{ in } \Omega.
$$
Firstly, let us show that $u_\infty =z=v$ in the set $ \{x\in \Omega \, : \, z (x) \geq  0 \}$. To this end, we observe that in the set $ \{ x\in \Omega \, : \, \nabla \mathfrak{h} (x) \geq 1 \}$ we have $
z(x) = u_\infty (x) = \mathfrak{h}(x) $. Hence, we have to deal with $ \{x\in \Omega \, : \, z (x) \geq  0 \mbox{ and } \nabla \mathfrak{h} (x) < 1 \}$.
Now, as in \cite[Theorem 4.2]{JPR}, we argue by contradiction and suppose that there is $\hat{x} \in  \{x\in \Omega \, : \, z (x) \geq  0 \mbox{ and } \nabla \mathfrak{h}(x) < 1 \}$ such that
$u_\infty(\hat x)-z(\hat x)>0$. If $u_\infty$ were smooth, we would have $|\nabla u_\infty(\hat
x)|\geq 1$ by the second part of the equation, and from
$\Delta_\infty u_\infty\ge 0$ it would follow that $t\mapsto |\nabla u_\infty(\gamma(t))|$
is non-decreasing along the curve $\gamma$ for which $\gamma(0)=\hat x$
and $\dot{\gamma}(t)=\nabla u_\infty(\gamma(t))$. Using this information and the fact that $|z(x)-z(y)|\le
|x-y|$ in $\{\nabla \mathfrak{h} < 1\}$, we could then follow $\gamma$ up to the boundary
to find a point $y$ where $u_\infty(y)>z(y)$; but this is a contradiction since
$u_\infty$ and $z$ coincide on $\partial \Omega$.

To overcome the lack of smoothness to $u_\infty$ and to justify rigorously the steps outlined above, we use an approximation procedure with the sup-convolution. Let $\delta>0$ and
\[
(u_\infty)_\delta (x)=\sup_{y\in\Omega} \left\{
u_\infty(y)-\frac1{2\delta} |x-y|^2\right\}
\]
be the standard sup-convolution of $u_\infty$. Observe that since $u_\infty$ is bounded in $\Omega$,
we in fact have
\[
(u_\infty)_\delta(x)=\sup_{y\in B_{R(\delta)}(x)} \left\{
u_\infty (y)-\frac1{2\delta} |x-y|^2\right\}
\]
with $R(\delta)=2\sqrt{\delta\|u_\infty \|_{L^\infty(\Omega)}}$. We
assume that $\delta>0$ is small.
In what follows we will use the notation
\begin{equation*}
 L(f, x) := \lim_{r \to +0} \mathrm{Lip}(f, B_r(x))
\end{equation*}
for the point-wise Lipschitz constant of a function $f$.
Next we observe that since $u_\infty$ is a solution to
\eqref{pppp}, it follows that $\Delta_\infty
(u_\infty)_\delta\ge 0$ and $|\nabla (u_\infty)_\delta|-\chi_{(u_\infty)_\delta > 0}\ge 0$.
In particular, since $(u_\infty)_\delta$ is semi-convex, there
exists $x_0$ such that $$(u_\infty)_\delta(x_0)-z(x_0) > \sup\limits_{x\in
\partial\Omega} ((u_\infty)_\delta-z),$$
and
$$
|\nabla (u_\infty)_\delta(x_0)|=L(u_\delta,x_0)\ge
1.
$$
 Now let $r_0=\frac12\mbox{dist}(x_0,\partial\Omega)$ and
let $x_1\in \partial B_{r_0}(x_0)$ be a point such that
\[
\max_{y\in \overline B_{r_0}(x_0)} (u_\infty)_\delta(y)=(u_\infty)_\delta(x_1).
\]
Since $\Delta_\infty (u_\infty)_\delta\ge 0$, the increasing slope estimate, see \cite{ceg}, implies
\[
1\le L((u_\infty)_\delta,x_0) \le L((u_\infty)_\delta,x_1)\quad \text{and}\quad (u_\infty)_\delta(x_1)
\ge (u_\infty)_\delta(x_0)+ |x_0-x_1|.
\]
By defining $r_1=\frac12\dist(x_1,\partial\Omega)$,
choosing $x_2\in \partial B_{r_1}(x_1)$ so that
$$\max\limits_{y\in \overline B_{r_1}(x_1)}
(u_\infty)_\delta(y)=(u_\infty)_\delta(x_2),$$ and using the increasing slope
estimate again yields
\[
1\le L((u_\infty)_\delta,x_0) \le L((u_\infty)_\delta,x_1)\le L((u_\infty)_\delta,x_2)
\]
and
\[
(u_\infty)_\delta(x_2) \ge  (u_\infty)_\delta(x_1)+ |x_1-x_2| \ge
(u_\infty)_\delta(x_0)+ |x_0-x_1|+ |x_1-x_2|.
\]
Repeating this construction we obtain a sequence $(x_k)$ such that
$x_k\to a\in \partial  \{x\in \Omega \, : \, z (x) \geq  0 \mbox{ and }
\nabla \mathfrak{h} (x) < 1 \} \cap \partial \Omega $
as $k\to\infty$ and
\[
(u_\infty)_\delta(x_k)\ge (u_\infty)_\delta(x_0)+\sum_{j=0}^{k-1} |x_j-x_{j+1}|
\quad\text{for $k=1,2,\ldots$}
\]
On the other hand, since $|z(x)-z(y)|\le |x-y|$ whenever the line segment
$[x,y]$ is contained in $\{\nabla \mathfrak{h} \leq 1\}$ (see \cite{cgw}),
we have
\[
z(x_k) \le z(x_0)+\sum_{j=0}^{k-1} |x_j-x_{j+1}|.
\]
Thus, by continuity,
\[
(u_\infty)_\delta(a)-z(a) =\lim_{k\to\infty} (u_\infty)_\delta(x_k)-z(x_k) \ge (u_\infty)_\delta(x_0)-z(x_0) >
\sup\limits_{x\in \partial\Omega} ((u_\infty)_\delta-z),
\]
which is clearly a contradiction. Therefore, we conclude that $u_\infty =z=v$ in the set $ \{x\in \Omega \, : \, z (x) \geq  0 \}$.

To extend the equality $u_\infty = v$ to the set $ \{x\in \Omega \, : \, z (x) <  0 \}$
we just observe that $-\Delta_\infty v = 0$ there and also that $-\Delta_\infty u_\infty =0$ since
$u_\infty \leq 0$ on the boundary of $ \{x\in \Omega \, : \, z (x) <  0 \}$ and then $u_\infty \leq 0$
in the set $ \{x\in \Omega \, : \, z (x) <  0 \}$ (notice that if $u_\infty =0$ there then trivially
$-\Delta_\infty u_\infty =0$). Therefore, we conclude that
$$ u_\infty = v$$
in the whole $\Omega$.
\end{proof}

\begin{remark} {\rm From the previous proof we have that the positivity sets of $u_\infty$ and $z$ coincide. The function $z$ can be computed as follows (see \cite[Section 2.2]{JPR}):
Since $\mathfrak{h}$ is everywhere differentiable, see \cite{es}, and $|\nabla \mathfrak{h}(x)|$ equals to the point-wise Lipschitz constant of $\mathfrak{h}$,
\begin{equation*}
 L(\mathfrak{h}, x) \defeq \lim_{r \to +0} \mathrm{Lip}(\mathfrak{h}, B_r(x))
\end{equation*}
for every $x\in\Omega$, using that the map $x\mapsto L(h,x)$ is upper semicontinuous, see e.g.\ \cite{ceg},
we have that the set
$$
  V \defeq \{x \in \Omega\colon |\nabla \mathfrak{h}(x)| < 1\}
$$
is an open subset of $\Omega$. Now, define the ''patched function``
$z \colon\overline\Omega \to\R$ by first setting
$$
z=\mathfrak{h}\quad \mbox{ in }\overline\Omega\setminus V,
$$
and then, for each connected component $U$ of $V$ and $x\in
U$, we let
$$
 z(x)
  = \sup_{y \in \partial U}
\left(
 \mathfrak{h}(y) - d_{U} (x, y)
\right),
$$
where $d_{U} (x, y)$ stands for the (interior) distance between $x$ and $y$ in $U$.
}
\end{remark}

\section{Games: Pay or Leave Tug-of-War}\label{sec:games}

In this section, we consider a variant of the Tug-of-War games introduced in \cite{PSSW} and \cite{JPR}. Let us describe the two-player zero-sum game that we call \textit{Pay or Leave Tug-of-War}.

Let $\Om$ be a bounded open set and $\eps>0$. A token is placed at $x_0\in\Om$.
Player~II, the player seeking to minimize the final payoff, can either pass the turn to Player~I
or decide to toss a fair coin and play Tug-of-War. In this case, the winner of the coin toss gets to move the token to any $x_1\in B_\eps (x_0)$. If Player II pass the turn to Player~I, then she can
either move the game token to any $x_1\in B_\eps (x_0)$ with the price $-\eps$ or decide to end the game immediately with no payoff for either of the players. After the first round, the game continues from $x_1$ according to the same rules.

This procedure yields a possibly infinite sequence of game states $x_0,x_1,\ldots$ where every $x_k$ is a random variable. If the game is not ended by the rules described above, the game
ends when the token leaves $\Omega$, at this point the token will be in the boundary strip of
width $\eps$ given by
$$
\Gamma_\eps=
\{x\in \R^n \setminus \Om\,:\,\dist(x,\partial \Om)< \eps\}.
$$
We denote by $x_\tau \in \Gamma_\eps$ the first point in the sequence of game states that lies in $\Gamma_\eps$, so that $\tau$ refers to the first time we hit $\Gamma_{\eps}$.

At this time the game ends with the terminal payoff given by $F(x_\tau)$, where
$F:\Gamma_\eps
\to
\R$ is a given Borel measurable continuous \emph{payoff function}. Player I earns $F(x_\tau)$ while Player II earns $-F(x_\tau)$.

A strategy $S_\I$ for Player I is a function defined on the partial histories that gives the next game position $S_\I{\left(x_0,x_1,\ldots,x_k\right)}=x_{k+1}\in  B_\eps(x_k)$
if Player I gets to move the token. Similarly Player II plays according to
a strategy $S_\II$. In addition, we define a decision variable for Player II,
which tells when Player II decides to pass a turn
\[
\begin{split}
\theta_{\II}(x_0,\ldots,x_k)=
\begin{cases}
1,&\trm{Player
II pass a turn,}\\
0,& \trm{otherwise},
\end{cases}
\end{split}
\]
and one for Player I which tells when Player I decides to end the game immediately
\[
\begin{split}
\theta_{\I}(x_0,\ldots,x_k)=
\begin{cases}
1,&\trm{Player
I ends the game,}\\
0,& \trm{otherwise}.
\end{cases}
\end{split}
\]

Given the sequence $x_0,\ldots,x_k$ with $x_k\in\Om$ the game will end immediately when
\[
\theta_{\I}(x_0,\ldots,x_k)=\theta_{\II}(x_0,\ldots,x_k)=1.
\]
Otherwise, the one step transition probabilities will be
$$
\begin{array}{rcl}
  \pi_{S_\I,S_\II,\theta_{\I},\theta_{\II}}(x_0,\ldots,x_k,{A}) & = &
  \displaystyle \big(1-\theta_{\II}(x_0,\ldots,x_k)\big)\frac{1}{2}\Big(
\delta_{S_\I(x_0,\ldots,x_k)}({A})+\delta_{S_\II(x_0,\ldots,x_k)}({A})\Big) \\[10pt]
   & & + \theta_{\II}(x_0,\ldots,x_k)(1-\theta_{\I}(x_0,\ldots,x_k))\delta_{S_\I(x_0,\ldots,x_k)}(A).
\end{array}
$$
By using the Kolmogorov's extension theorem and the one step transition probabilities, we can build a
probability measure $\mathbb{P}^{x_0}_{S_\I,S_\II,\theta_\I,\theta_\II}$ on the
game sequences. The expected payoff, when starting from $x_0$ and
using the strategies $S_\I,S_\II,\theta_\I,\theta_\II$, is
\begin{equation}
\label{eq:defi-expectation}
\begin{split}
&\mathbb{E}_{S_{\I},S_\II,\theta_\I,\theta_\II}^{x_0}\left[F(x_\tau)-\eps\sum_{k=0}^{\tau-1}
\theta_{\II}(x_0,\ldots,x_k)(1-\theta_{\I}(x_0,\ldots,x_k))\right]\\
&=\int_{H^\infty} \Big(F(x_\tau)-\eps
\sum_{k=0}^{\tau-1} \theta_{\II}(x_0,\ldots,x_k)(1-\theta_{\I}(x_0,\ldots,x_k))\Big) \ud
\mathbb{P}^{x_0}_{S_\I,S_\II,\theta_\I,\theta_\II},
\end{split}
\end{equation}
where $F\colon\Gamma_\eps\to\R$ is a given continuous function
prescribing the terminal payoff extended as $F\equiv 0$ in $\Om$.

The \emph{value of the game for Player I} is given by
\[
u_\I(x_0)=\sup_{S_{\I},\theta_\I}\inf_{S_\II,\theta_\II}\,
\mathbb{E}_{S_{\I},S_\II,\theta_\I,\theta_\II}^{x_0}\left[F(x_\tau)-\eps\sum_{i=0}^{\tau-1}
\theta_{\II}(x_0,\ldots,x_k)(1-\theta_{\I}(x_0,\ldots,x_k))\right]
\]
while the \emph{value of the game for Player II} is given by
\[
u_\II(x_0)=\inf_{S_\II,\theta_\II}\sup_{S_{\I},\theta_\I}\,
\mathbb{E}_{S_{\I},S_\II,\theta_\I,\theta_\II}^{x_0}\left[F(x_\tau)-\eps\sum_{i=0}^{\tau-1}
\theta_{\II}(x_0,\ldots,x_k)(1-\theta_{\I}(x_0,\ldots,x_k))\right].
\]
Intuitively, the values $u_\I(x_0)$ and $u_\II(x_0)$ are the best
expected outcomes
 each player can  guarantee when the game starts at
$x_0$. Observe that if the game does not end almost surely, then
the expectation \eqref{eq:defi-expectation} is undefined. In this
case, we define $\mathbb{E}_{S_{\I},S_\II,\theta_\I,\theta_\II}^{x_0}$ to take
value $-\infty$ when evaluating $u_\I(x_0)$ and $+\infty$
when evaluating $u_\II(x_0)$.
If $u_\I= u_\II$, we say that the game has a
value.

\subsection{The game value function and its Dynamic Programming Principle}\label{sec:games.6}

In this section, we prove that the game has a value, i.e., $u \defeq u_\I= u_\II$ and that such a
value function satisfies the Dynamic Programming Principle (DPP) given by
\[
\begin{split}
u(x) = \min \Bigg\{ \frac12  \left(\sup_{y\in  B_\eps (x)}
u(y) + \inf_{y\in  B_\eps (x)} u(y) \right) ;
\max \Bigg\{0;
 \sup_{y\in  B_\eps (x)} u(y) - \eps \Bigg\}\Bigg\}
\end{split}
\]
for $x\in\Om$ and $u(x)=F(x)$ for $x\in\Gamma_\eps$.

Let us see intuitively why this holds.
At each step, with the token in a given $x\in\Omega$, we have that Player II chooses whether to play Tug-of-War or to pass the turn to Player I.
In the first case with probability $\frac12$, Player I gets to move and will try to maximize the
expected outcome; and with probability $\frac12$, Player II gets to move and will try to minimize the
expected outcome. In this case the expected payoff will be
\[
\frac12 \sup_{y\in  B_\eps (x)}u(y) +\frac12 \inf_{y\in  B_\eps (x)} u(y).
\]
On the other hand, if Player II pass the turn to Player I, she will have two options.
To end the game immediately obtaining 0 or to move trying to maximize the
expected outcome by paying $\eps$. Player I will prefer the option that gives the greater payoff, that is, the expected payoff is given by
\[
\max \Bigg\{0;\sup_{y\in  B_\eps (x)} u(y) - \eps \Bigg\}
\]
Finally, Player II will decide between the two possible payoff mentioned here, preferring the one with the minimum payoff.

To prove that the DPP holds for our game, we borrow some ideas from \cite{BPR} and \cite{JPR}.
We choose a path that allows us to make the presentation self-contain.

We define $\Om_\eps=\Om \cup \Gamma_\eps$ and
$u_n:\Om_\eps\to\R$ a sequence of functions.
We define the sequence inductively, let $u_n=g$ on $\Gamma_\eps$, $$u_0=\max_{\Gamma_\eps}F$$ on $\Om$ and
\[
\begin{split}
u_{n+1}(x) = \min \Bigg\{ \frac12  \left(\sup_{B_\eps (x)}
u_n + \inf_{B_\eps (x)} u_n \right) ;
\max \Bigg\{0;
 \sup_{B_\eps (x)} u_n - \eps \Bigg\}\Bigg\}
\end{split}
\]
on $\Om$ for all $n\in\N$.

Let us observe that $u_0\geq u_1$ and, in addition, if $u_{n-1}\geq u_n$, by the recursive definition, we have $u_n\geq u_{n+1}$.
Then, by induction, we obtain that the sequence of functions is a decreasing sequence.
By the definition we have that the sequence is bounded by below by $\min\big\{0,\min_{\Gamma_\eps}F\big\}$.
Hence, $u_n$ converge point-wise to a bounded Borel function $u$.

We want to prove that the limit $u$ satisfy the dynamic programming principle.
We can attempt to do that by passing to the limit in the recursive formula.
Since $u_n$ is a decreasing sequence that converge point-wisely to $u$, we can show that
\[
\inf_{B_\eps (x)} u_n \to \inf_{B_\eps (x)} u.
\]
Although, this convergence is not immediate for the supremum.
This is why, in order to be able to pass to the limit in the recursive formula, we want to show that the sequence converges uniformly. To this end, let us prove an auxiliary lemma.

\begin{lemma}
\label{lem:lambda}
Let $x\in\Omega$, $n\in\N$ and fix $\lambda_1$, $\lambda_2$ and $\delta$ such that
\[
u_n(x)-u_{n+1}(x)\geq \lambda_1,\quad \|u_{n-1}-u_n\|_\infty\leq \lambda_2
\]
and $\delta >0$. Then there exists $y\in B_\eps(x)$ such that
\[
\lambda_2-2\lambda_1+\delta +u_{n-1}(y) \geq \sup_{B_\eps (x)}u_n,
\]
\end{lemma}
\begin{proof}
Given $\lambda_1\leq u_n(x)-u_{n+1}(x)$, by the recursive definition, we have
\begin{align*}
\lambda_1\leq&
\min \Bigg\{ \frac12  \Bigg(\sup_{ B_\eps (x)}
u_{n-1}+  \inf_{ B_\eps (x)} u_{n-1} \Bigg);
\max \Bigg\{0;
 \sup_{B_\eps (x)} u_{n-1}- \eps \Bigg\}\Bigg\}\\
&-
\min \Bigg\{ \frac12  \Bigg(\sup_{B_\eps (x)}u_n+ \inf_{B_\eps (x)} u_n \Bigg) ;
\max \Bigg\{0;
 \sup_{B_\eps (x)} u_n- \eps \Bigg\}\Bigg\}.
\end{align*}
From the standard inequalities
$$
\min\{a,b\}-\min\{c,d\}\leq \max\{a-c,b-d\} \quad \text{and} \quad
\max\{a,b\}-\max\{c,d\}\leq \max\{a-c,b-d\},
$$
we get
$$
\lambda_1\leq
 \max \Bigg\{
\frac12  \Bigg(\sup_{B_\eps (x)}u_{n-1}
+  \inf_{B_\eps (x)} u_{n-1} \Bigg)
-\frac12  \Bigg(\sup_{B_\eps (x)}u_n
+ \inf_{B_\eps (x)} u_n \Bigg);0;
\sup_{B_\eps (x)}u_{n-1}
-\sup_{B_\eps (x)}u_n
\Bigg\}.\\
$$
Since $u_{n-1}\geq u_n$ we can avoid the term 0 in the RHS, we obtain
$$
\lambda_1\leq
\frac12 \Bigg(
\sup_{B_\eps (x)}u_{n-1}
-\sup_{B_\eps (x)}u_n
\Bigg)
+\frac12 \max \Bigg\{
\inf_{B_\eps (x)} u_{n-1}
- \inf_{B_\eps (x)} u_n;
\sup_{B_\eps (x)}u_{n-1}
-\sup_{B_\eps (x)}u_n
\Bigg\}.
$$
We bound the difference between the suprema and infima using the inequality
$\|u_{n-1}-u_n\|_\infty\leq \lambda_2$, we obtain
\[
2\lambda_1\leq\Bigg(\sup_{B_\eps (x)}u_{n-1}
-\sup_{B_\eps (x)}u_n \Bigg)+\lambda_2,
\]
that is,
\[
2\lambda_1-\lambda_2 +\sup_{B_\eps (x)}u_n \leq \sup_{B_\eps (x)}u_{n-1}.
\]

Finally, we can choose $y\in B_\eps(x)$ such that
\[
u_{n-1}(y)+\delta\geq \sup_{B_\eps (x)}u_{n-1}
\]
which gives the desired inequality.
\end{proof}

\begin{proposition}
\label{prop:unifconv}
The sequence $u_n$ converges uniformly and the limit $u$ is a solution to the DPP.
\end{proposition}

\begin{proof}
We want to show that the convergence is uniform. Suppose not. Observe that if
$||u_n-u_{n+1}||_\infty\to 0$ we can extract a uniformly Cauchy subsequence, thus this
subsequence converges uniformly to a limit $u$. This implies that the $u_n$ converge
uniformly to $u$, because of the monotonicity. By the recursive definition we have
$\|u_n-u_{n+1}\|_\infty\geq \|u_{n-1}-u_n \|_\infty\geq 0$.
Then, as we are assuming the convergence
is not uniform, we have
\[
\|u_n-u_{n+1}\|_\infty\to M \quad \text{and} \quad \|u_n-u_{n+1}\|_\infty\geq M
\]
for some $M>0$.

Given $\delta>0$, let $n_0\in\N$ such that for all $n\geq n_0$,
\[
\|u_n-u_{n+1}\|_\infty\leq M+\delta.
\]
We fix $k\in\N$. Let $x_0\in\Om$ such that
\[
M-\delta<u_{n_0+k-1}(x_0)-u_{n_0+k}(x_0).
\]

Now we apply Lemma~\ref{lem:lambda} for $n=n_0+k-1$, $\lambda_1=M-\delta$ and $\lambda_2=M+\delta$, we get
\begin{align*}
u_{n_0+k-1}(x_0),u_{n_0+k-1}(x_1)&\leq \sup_{B_\eps(x_0)}u_{n_0+k-1}\\
&\leq u_{n_0+k-2}(x_1)+\lambda_2-2\lambda_1+\delta\\
&\leq u_{n_0+k-2}(x_1)+4\delta-M
\end{align*}
for some $x_1\in B_\eps(x_0)$.
If we repeat the argument for $x_1$, but now with $\lambda_1=t\delta-M$, we obtain
\[
u_{n_0+k-2}(x_1),u_{n_0+k-2}(x_2)\leq u_{n_0+k-3}(x_2)+(2t+2)\delta-M.
\]
Inductively, we obtain a sequence $x_l$, $1\leq l \leq k-1$ such that
\[
u_{n_0+k-l}(x_{l-1}),u_{n_0+k-l}(x_l)\leq u_{n_0+k-l-1}(x_l)+(3\times 2^l-2)\delta-M.
\]
If we add the inequalities
\[
u_{n_0+k-l}(x_{l-1})\leq u_{n_0+k-l-1}(x_l)+(3\times 2^l-2)\delta-M
\]
for $1\leq l \leq k-1$ and $u_{n_0+k}(x_0)\leq u_{n_0+k-1}(x_0)+\delta-M$, we get
\[
u_{n_0+k}(x_0)-u_{n_0}(x_{k-1})\leq (3\times 2^k-2k-3)\delta-kM.
\]
Which is a contradiction since $u_n$ is bounded but we can make the RHS as small as we want by choosing a big value for $k$ and a small one for $\delta$.
\end{proof}

Now, we are ready to prove one of the main results of this section.

\begin{theorem}[{\bf Dynamic Programming Principle}]
\label{teo:DPP}
The game has a value $u=u_\I= u_\II$ and it satisfies
\[
\begin{split}
u(x) = \min \Bigg\{ \frac12  \left( \sup_{y\in  B_\eps (x)}
u(y) + \inf_{y\in  B_\eps (x)} u(y)\right) ;
\max \Bigg\{0;
 \sup_{y\in  B_\eps (x)} u(y) - \eps \Bigg\}\Bigg\}
\end{split}
\]
for $x\in \Om$ and  $u(x)=F(x)$ in $\Gamma_\eps$.
\end{theorem}

\begin{proof}
By definition, $u_\I\le u_\II$.
We will show that $u_\II\leq u$ and $u\leq u_\I$ for the $u$ constructed in Proposition~\ref{prop:unifconv}.
This, together with the fact that $u$ satisfy the DPP will complete the proof.
For the first inequality we will use the constructed sequence of function $u_n$ as in \cite{BPR}.
For the second inequality we will use an argument similar to one in \cite{JPR}.

We want to show that $u_\II\leq u$.
Given $\eta>0$ let $n>0$ be such that $u_n(x_0)<u(x_0)+\frac{\eta}{2}$.
We build an strategy ($S^0_\II, \theta^0_\II$) for Player II, in the firsts $n$ moves, given $x_{k-1}$ she will choose to
play Tug-of-War or pass the turn depending whether
$$
\frac12 \left(\inf_{B_\eps(x_{k-1})} u_{n-k}+\sup_{B_\eps(x_{k-1})} u_{n-k}\right) \quad \text{or}\quad
\max\Bigg\{0;\sup_{B_\eps(x_{k-1})} u_{n-k}-\eps\Bigg\}
$$
is larger.
When playing Tug-of-War she will move to a point that almost minimize $u_{n-k}$, that is, she chooses $x_k\in B_\eps(x_{k-1})$ such that
\[
u_{n-k}(x_k)<\inf_{B_\eps(x_{k-1})}u_{n-k}+\frac{\eta}{2n}.
\]
After the first $n$ moves she will choose to play Tug-of-War following a strategy that ends the game almost surely (for example pointing in a fix direction).

We have
\[
\begin{split}
&\mathbb{E}_{S^0_\I, S_\II}^{x_0}[u_{n-k}(x_k)+\frac{(n-k)\eta}{2n}|\,x_0,\ldots,x_{k-1}]
\\
&\leq\min \left\{\frac12 \Bigg( \inf_{B_\eps(x_{k-1})} u_{n-k}+\sup_{B_\eps(x_{k-1})} u_{n-k} + \frac{\eta}{n}\Bigg);\max\Bigg\{0;\sup_{B_\eps(x_{k-1})} u_{n-k}-\eps\Bigg\}\right\}+\frac{(n-k)\eta}{2n}
\\
&\leq u_{n-k+1}(x_{k-1})+\frac{(n-k+1)\eta}{2n},
\end{split}
\]
where we have estimated the strategy of Player I by $\sup$ and
used the construction for the $u_k$'s.
Thus
\[
M_k=
\left\{ \begin{array}{ll}
\displaystyle
u_{n-k}(x_k)+\frac{(n-k)\eta}{2n} \qquad & \mbox{
for $0\leq k\leq n$}, \\[10pt]
\displaystyle \sup_{\Gamma_{\varepsilon}} F \qquad & \mbox{ for $k>n$},
\end{array}
\right.
\]
is a supermartingale.

Now we have
\begin{equation}
\label{eq:martargument}
\begin{split}
u_\II(x_0)
&=\inf_{S_\II,\theta_\II}\sup_{S_{\I},\theta_\I}\,
\mathbb{E}_{S_{\I},S_\II,\theta_\I,\theta_\II}^{x_0}\left[F(x_\tau)-\eps\sum_{i=0}^{\tau-1}
\theta_{\II}(x_0,\ldots,x_i)(1-\theta_{\I}(x_0,\ldots,x_i))\right]\\
&\leq\sup_{S_{\I},\theta_\I}\,
\mathbb{E}_{S_{\I},S^0_\II,\theta_\I,\theta^0_\II}^{x_0}\left[F(x_\tau)-\eps\sum_{i=0}^{\tau-1}
\theta_{\II}(x_0,\ldots,x_i)(1-\theta_{\I}(x_0,\ldots,x_i))\right]\\
&\leq \inf_{S_\II} \liminf_{k\to\infty}\mathbb{E}_{S_{\I},S^0_\II,\theta_\I,\theta^0_\II}^{x_0}[M_{\tau\wedge k}]\\
&\leq\inf_{S_\II}  \mathbb{E}_{S_{\I},S^0_\II,\theta_\I,\theta^0_\II}^{x_0}[M_0]=u_n(x_0)+\frac{\eta}{2}<u(x_0)+\eta,
\end{split}
\end{equation}
where $\tau\wedge k \defeq \min\{\tau,k\}$, and we used the optional stopping theorem for $M_{k}$.
Since $\eta$ is arbitrary this proves the claim.

Now, we will show that $u\leq u_\I$.
We want to find a strategy ($S^0_\I, \theta^0_\I$) for Player I, that ensures a payoff close to $u$.
He has to maximize the expected payoff and, at the same time, make sure that the game ends almost sure.
This is done by using the backtracking strategy (cf. \cite[Theorem 2.2]{PSSW} for more details).

To that end, we define
\[
\delta(x)=\sup_{B_\eps(x)} u -u(x).
\]
Fix $\eta>0$ and a starting point $x_0\in \Om$, and set $\delta_0
=\min\{\delta(x_0),\eps\}/2$. We suppose for now that $\delta_0>0$, and
define
\[
\begin{split}
X_0=\Big\{x\in \Om\,:\, \delta(x)> \delta_0\Big\}.
\end{split}
\]
We consider a strategy $S_I^0$ for Player I that distinguishes between the cases $x_k\in X_0$
and $x_k\notin X_0$.
To that end, we define
\[
m_k=
\begin{cases}
u(x_k)-\eta 2^{-k} &\text{ if } x_k\in X_0\\
u(y_k)-\delta_0 d_k-\eta 2^{-k} &\text{ if } x_k \notin X_0
\end{cases}
\]
and
\[
M_k=m_k-\eps\sum_{i=0}^{k-1}
\theta_{\II}(x_0,\ldots,x_i)(1-\theta_{\I}(x_0,\ldots,x_i))
\]
where $y_k$ denotes the last game position in $X_0$ up to time $k$, and $d_k$ is the distance, measured in
number of steps, from $x_k$ to $y_k$ along the graph spanned by the previous points
$y_k=x_{k-j},x_{k-j+1},\ldots,x_k$ that were used to get from $y_k$ to $x_k$.

In what follows we define a strategy for Player I and prove that $M_k$ is a submartingale.
Observe that $M_{k+1}-m_{k+1}=M_k-m_k$ or $M_{k+1}-m_{k+1}=M_k-m_k-\eps$, so to prove the desired submartingale property we will mostly make computations in terms of $m_k$.

First, if $x_k\in X_0$, then Player I chooses to step to a point
$x_{k+1}$ satisfying
\[
u(x_{k+1})\ge \sup_{B_{\eps}(x_k)} u-\eta_{k+1} 2^{-(k+1)},
\]
where $\eta_{k+1}\in (0,\eta]$ is small enough to guarantee that $x_{k+1}\in X_0$.
Let us remark that
\begin{equation}
\label{eq:up-more-than-down}
u(x)-\inf_{B_\eps(x)}u \leq  \sup_{B_\eps(x)}u-u(x)=\delta(x)
\end{equation}
and hence
$$
\begin{array}{rcl}
  \delta(x_k)-\eta_{k+1} 2^{-(k+1)} & \leq & \displaystyle \sup_{B_\eps(x_k)} u -u(x_k)-\eta_{k+1} 2^{-(k+1)}
   \leq  u(x_{k+1})-u(x_k) \\
   & \leq & \displaystyle u(x_{k+1})-\inf_{B_\eps(x_{k+1})}u  \leq  \delta(x_{k+1}).
\end{array}
$$
Therefore, we can guarantee that $x_{k+1}\in X_0$ by choosing $\eta_{k+1}$ such that
\[
\delta_0<\delta(x_k)-\eta_{k+1} 2^{-(k+1)}.
\]

Thus if $x_k\in X_0$ and Player I gets to choose the next position, it holds that
 \[
 \begin{split}
 m_{k+1}&\ge u(x_k)+\delta(x_k)-\eta_{k+1} 2^{-(k+1)}-\eta 2^{-(k+1)}\\
 &\ge u(x_k)+\delta(x_k)-\eta 2^{-k}\\
 &=m_k+\delta(x_k).
 \end{split}
 \]

When Tug-of-War is played, if Player II wins the toss and moves from $x_k\in X_0$ to $x_{k+1}\in X_0$,  it
holds, in view of \eqref{eq:up-more-than-down}, that
\[
\begin{split}
m_{k+1}\ge u(x_k)-\delta(x_k)-\eta 2^{-(k+1)}>m_k-\delta(x_k). 
\end{split}
\]

If Player II wins the toss and she moves
to a point $x_{k+1}\notin X_0$ (whether  $x_k\in X_0$ or not), it holds that
\begin{equation}
\label{eq:ineqxk+1notinX_0}
\begin{split}
m_{k+1}&= u(y_k)-d_{k+1} \delta_0-\eta 2^{-(k+1)}\\
&\ge u(y_k)-d_{k} \delta_0-\delta_0-\eta 2^{-k}\\
&=m_k-\delta_0 .
\end{split}
\end{equation}

When Player II pass the turn to Player I, he can choose to end the game immediately or to move by paying $\eps$. If $\delta(x_k)\geq \eps$ he will choose to play, we get $M_{k+1}\geq M_k+\delta(x_k)-\eps\geq M_k$.
If $\eps>\delta(x_k)$, the DPP implies that $0\geq u(x_k)$ and hence he can finish the game immediately earning more than $m_k$.

In the case $x_k\notin X_0$, the strategy for Player I is to backtrack to $y_k$,
that is, if he wins the coin toss, he moves the token to one of the points
$x_{k-j},x_{k-j+1},\ldots,x_{k-1}$ closer to $y_k$ so that $d_{k+1}= d_k-1$.

Thus if Player I wins and
$x_k\notin X_0$ (whether $x_{k+1}\in X_0$ or not),
\[
\begin{split}
m_{k+1}\ge \delta_0+m_k.
\end{split}
\]

When Tug-of-War is played, if Player II wins the coin toss and moves from $x_k\notin X_0$ to $x_{k+1}\in X_0$, then
\[
\begin{split}
m_{k+1}=u(x_{k+1})-\eta 2^{-(k+1)}\ge -\delta(x_k)+u(x_k)-\eta 2^{-k}\ge -\delta_0+m_k
\end{split}
\]
where the first inequality is due to \eqref{eq:up-more-than-down}, and the second follows from the fact
$m_k=u(y_k)-d_k\delta_0-\eta 2^{-k}\le u(x_k)-\eta 2^{-k}$.
The same was obtained in \eqref{eq:ineqxk+1notinX_0} when $x_{k+1}\notin X_0$.

It remains to analyse what happens when Player II pass the turn to Player I in this case.
Since $\delta(x_k)\leq \eps/2<\eps$, we have $0\geq u(x_k)$ and as before he can finish the game immediately earning more than $m_k$.

Taking into account all the different cases, we see that $M_k$ is a submartingale. We can also see that when the game ends Player I ensures a payoff of al least $M_k$. Let us observe that $m_k$ is also a submartingale, and it is bounded. Since Player I can assure that $m_{k+1}\ge m_k+\delta_0$ if he gets to move the token, the game must terminate almost surely. this is because, there are arbitrary long
sequences of moves made by Player I (if he does not end the game immediately).
Indeed, if Player II pass a turn, then Player I gets to move, and otherwise this is a
consequence of the zero-one law.

We can now conclude the proof with an inequality analogous to that in \eqref{eq:martargument}.

Finally, let us remove the assumption that $\delta(x_0)>0$. If $\delta(x_0)=0$ for $x_0\in \Omega$, when Tug-of-War is played, Player I adopts a strategy of pulling towards a boundary point
until the game token reaches a point $x_0'$ such that $\delta(x_0')>0$ or $x_0'$ is outside
$\Om$. It holds that $u(x_0)= u(x_0')$, because by \eqref{eq:up-more-than-down}. If Player II passes the turn, Player I end the game immediately earning 0 (recall that $\delta(x)=0$ implies $0\geq u(x)$ because of the DPP).
\end{proof}

\subsection{Game value convergence}
\setcounter{equation}{0}

In this subsection we study the behaviour of the game values as $\eps \to 0$. In the previous sections we have analyse the game for a fix value of $\eps$, here we will consider the game value for different values of $\eps$. For this purpose, we will refer to the game value as $u^\eps$, emphasizing its dependence on $\eps$.
We want to prove that
\[
u^\eps\to u
\]
uniformly on $\ol\Omega$ as $\eps\to0$, and that $u$ is a viscosity solution to
\begin{equation}\label{eq:u>0}
\left\{
\begin{array}{rclcl}
     \max\{-\Delta_\infty u, \chi_{\{u>0\}}-\abs{\nabla u}\} & = & 0 & \mbox{in} & \Omega \\
     u(x) & = & F(x) & \mbox{on} & \partial \Omega,
\end{array}
\right.
\end{equation}

 To this end, we would like to apply the following Arzel\`{a}-Ascoli type lemma. We refer to the interested reader to \cite[Lemma 4.2]{MPR1} for a proof.

\begin{lemma}\label{lem.ascoli.arzela} Let $\{u^\eps : \overline{\Om}
\to \R,\ \eps>0\}$ be a set of functions such that
\begin{enumerate}
\item there exists $C>0$ such that $\abs{u^\eps(x)}<C$ for
    every $\eps>0$ and every $x \in \ol \Om$,
\item \label{cond:2} given $\eta>0$ there are constants
    $r_0$ and $\eps_0$ such that for every $\eps < \eps_0$
    and any $x, y \in \overline{\Om}$ with $|x - y | < r_0 $
    it holds
$$
|u^\eps (x) - u^\eps (y)| < \eta.
$$
\end{enumerate}
Then, there exists a uniformly continuous function $u:
\overline{\Om} \to \R$ and a subsequence still denoted by
$\{u^\eps\}$ such that
\[
\begin{split}
u^{\eps}\to u \qquad\textrm{ uniformly in}\quad\overline{\Om},
\end{split}
\]
as $\eps\to 0$.
\end{lemma}

So our task now is to show that the family $u^\eps$ satisfies the hypotheses of the previous lemma.
In the next Lemma, we prove that the family is asymptotically uniformly continuous, that is, it satisfies the condition \ref{cond:2} on Lemma~\ref{lem.ascoli.arzela}.
To do that we follow \cite{JPR}.

\begin{lemma}
\label{lem:asymp}
The family $u^\eps$ is asymptotically uniformly continuous.
\end{lemma}

\begin{proof}
We prove the required oscillation estimate by arguing by contradiction:
We define
$$
A(x) \defeq \sup_{B_\eps (x)} u^\eps - \inf_{ B_\eps (x)} u^\eps
$$

We claim that
$$
A(x) \leq 4 \max \{ \lip (F) ; 1 \} \eps,
$$
for all $x\in \Om$.
Aiming for a contradiction, suppose that there exists $x_0\in
\Omega$ such that
$$
A(x_0) > 4 \max \{ \lip (F) ; 1 \} \eps.
$$
In this case, we have that
\begin{equation}
\label{eq:tug-of-war-like-dpp}
\begin{split}
u^\eps(x_0)
&= \min \Bigg\{ \frac12  \Bigg( \sup_{B_\eps (x_0)}
u^\eps + \inf_{B_\eps (x_0)} u^\eps \Bigg);
\max \Bigg\{0;
 \sup_{B_\eps (x_0)} u^\eps - \eps \Bigg\}\Bigg\}\\
&= \frac12 \Bigg(\sup_{ B_\eps (x_0)}u^\eps  + \inf_{  B_\eps (x_0)} u^\eps \Bigg).
\end{split}
\end{equation}
The reason is that the alternative
\[
\begin{split}
 \frac12 \Bigg(
\sup_{ B_\eps (x_0)}u^\eps  + \inf_{  B_\eps (x_0)}
u^\eps \Bigg) &>
\max \Bigg\{0;
 \sup_{ B_\eps (x_0)} u^\eps - \eps \Bigg\}\\
 &>\sup_{ B_\eps (x_0)} u^\eps - \eps
\end{split}
\]
would imply
\begin{equation}
\label{eq:contradiction}
\begin{split}
A(x_0) =\sup_{ B_\eps (x_0)}u^\eps  - \inf_{
B_\eps (x_0)}u^\eps  < 2\eps,
\end{split}
\end{equation}
which is a contradiction with $A(x_0) > 4 \max \{ \lip (F) ; 1 \} \eps$.
It follows from \eqref{eq:tug-of-war-like-dpp} that
$$
\sup_{
 B_\eps (x_0)}u^\eps  -u^\eps (x_0) =u^\eps(x_0) -\inf_{  B_\eps
(x_0)}u^\eps  =\frac12 A(x_0) .
$$
Let $\eta>0$ and take $x_1 \in  B_\eps (x_0)$ such
that
$$
u^\eps (x_1) \geq \sup_{  B_\eps (x_0)}u^\eps  -
\frac{\eta}{2}.
$$
We obtain
\begin{equation}
\label{eq:twopoints}
u^\eps (x_1) -u^\eps (x_0) \geq
\frac12 A(x_0)-
\frac{\eta}{2} \geq 2\max \{ \lip (F) ; 1 \} \eps - \frac{\eta}{2},
\end{equation}
and, since $x_0\in  B_\eps (x_1)$, also
$$
\sup_{  B_\eps (x_1)}u^\eps  - \inf_{  B_\eps (x_1)}u^\eps   \geq 2\max
\{ \lip (F) ; 1 \} \eps - \frac{\eta}{2}.
$$
Arguing as before, \eqref{eq:tug-of-war-like-dpp} also holds at $x_1$, since otherwise the above
inequality would lead to a contradiction similarly as \eqref{eq:contradiction} for small enough $\eta$.

Thus, \eqref{eq:twopoints} and \eqref{eq:tug-of-war-like-dpp} imply
$$
\sup_{
 B_\eps (x_1)}u^\eps  -u^\eps (x_1)
= u^\eps (x_1)- \inf_{  B_\eps (x_1)}u^\eps \geq 2\max
\{ \lip (F) ; 1 \} \eps - \frac{\eta}{2},
$$
so that
$$
A(x_1) =
\sup_{
 B_\eps (x_1)}u^\eps -u^\eps (x_1) +u^\eps (x_1) - \inf_{  B_\eps (x_1)}
u^\eps
\geq 4\max
\{ \lip (F) ; 1 \} \eps - \eta.
$$
Iterating this procedure, we obtain $x_i \in  B_\eps (x_{i-1})$
such that
\begin{equation}\label{big_slope}
u^\eps (x_i) -u^\eps (x_{i-1}) \geq 2\max \{ \lip (F) ; 1
\}
\eps -
\frac{\eta}{2^i}
\end{equation}
and
\begin{equation}\label{big_osc}
A(x_i) \geq 4\max
\{ \lip (F) ; 1 \} \eps - \sum_{j=0}^{i-1}  \frac{\eta}{2^j}.
\end{equation}

We can proceed with an analogous argument considering points where
the infimum is nearly attained to obtain $x_{-1}$, $x_{-2}$,...
such that $x_{-i} \in  B_\eps (x_{-(i-1)})$, and
\eqref{big_slope} and \eqref{big_osc} hold.
Since $u^\eps$ is bounded, there must exist $k$ and $l$ such that
$x_k, x_{-l}\in\Gamma_\eps$, and we have
$$
\displaystyle \frac{|F (x_k) - F(x_{-l}) | }{| x_k - x_{-l}|}  \displaystyle  \geq
\frac{\displaystyle \sum\limits_{j=-l+1}^k u^\eps (x_{j}) -u^\eps (x_{j-1}) }{\eps
(k+l)} \displaystyle
\geq 2\max \{ \lip (F) ; 1
\} -
\frac{2\eta}{\eps},
$$
a contradiction.
Therefore
$$
A(x) \leq 4 \max \{\lip (F) ; 1 \} \eps,
$$
for every $x\in \Omega$.
\end{proof}

\begin{lemma}\label{lem:convergence}
Let $u^\eps$ be a family of game values for a Lipschitz continuous
boundary data $F$. Then there exists a Lipschitz continuous
function $u$ such that, up to selecting a subsequence,
\[
\begin{split}
u^\eps \to u \quad \trm{uniformly in }\ol \Om
\end{split}
\]
as $\eps\to0$.
\end{lemma}

\begin{proof}
By choosing always to play Tug-of-War and moving with any strategy that ends the game almost sure (as pulling in a fix direction), Player~II can ensure that the final payoff is at most $\max_{\Gamma_\eps}F$.
Similarly, by ending the game immediately if given the option and moving with any strategy that ends the game almost sure when playing Tug-of-War, Player~II can ensure that the final payoff is at least $\displaystyle \min\{0,\min_{\Gamma_\eps}F\}$.
We have
\[
 \displaystyle \min\{0,\min_{\Gamma_\eps} F\}\leq u^\eps \leq \max_{\Gamma_\eps}F.
\]

This, together with Lemma~\ref{lem:asymp}, shows that the family $u^\eps$ satisfies the hypothesis
of Lemma~\ref{lem.ascoli.arzela}.
\end{proof}

\begin{theorem} \label{thm:sol-viscosity}
The function u obtained as a limit in Lemma~\ref{lem:convergence}
is a viscosity solution to \eqref{eq:u>0}.
\end{theorem}

\begin{proof}
First, we observe that $u=F$ on $\partial \Omega$ due to
$u_\eps=F$ on $\partial \Omega$ for all $\eps>0$. Hence, we can focus
our attention on showing that $u$ satisfy the equation inside $\Omega$ in the
viscosity sense.

To this end, we obtain the following asymptotic expansions, as in \cite{MPR}.
Choose a point $x\in \Omega$ and a $C^2$-function  $\psi$ defined in a
neighbourhood of $x$. Note that since $\psi$ is continuous then we have
\[
\min_{\ol B_{\eps}(x)} \psi = \inf_{ B_{\eps}(x)} \psi
\quad\quad\text{ and }\quad\quad
\max_{\ol B_{\eps}(x)} \psi = \sup_{ B_{\eps}(x)} \psi
\]
for all $x\in\Omega$.
Let $x_1^\eps$ and $x_2^\eps$ be a minimum point and a maximum point, respectively,
for $\phi$ in $\ol B_\eps(x)$.
It follows from the Taylor expansions in \cite{MPR} that
\begin{equation}\label{eq:asymp-inf-lap}
\frac12\left(\max_{y\in \ol B_\eps (x)}
\psi (y) + \min_{y\in \ol B_\eps(x)} \psi
(y)\right)-\psi(x)\ge \eps^2\Big\langle
D^2\psi(x)\left(\frac{x_1^{\eps}-x}{\eps}\right),
\left(\frac{x_1^{\eps}-x}{\eps}\right)\Big\rangle +o(\eps^2).
\end{equation}
and
\begin{equation}
\label{eq:asymp-jensen}
\max_{y\in  \ol B_\eps (x)}
\phi (y) -\eps-
\phi(x)\ge \left(D \phi(x)\cdot
\tfrac{x_2^\eps-x}{\eps}-1\right)\eps
+\frac{\eps^2}{2}
D^2\phi(x)\left(\tfrac{x_2^{\eps}-x}{\eps}\right)\cdot
\left(\tfrac{x_2^{\eps}-x}{\eps}\right)+o(\eps^2).
\end{equation}

Suppose that $u-\psi$ has an strict local minimum. We want to prove that
\[
\max\{-\Delta_\infty \psi(x), \chi_{\{u\geq  0\}}(x)-\abs{\nabla\psi(x)}\}\ge 0.
\]
If $\nabla\psi(x)=0$, we have $-\Delta_\infty \psi(x)=0$ and hence the inequality holds.
We can assume $\nabla\psi(x)\neq 0$. By the uniform convergence, there exists sequence
$x_{\eps}$ converging to $x$ such that $u^{\eps} - \psi $ has
an approximate minimum at $x_{\eps}$, that is, for $\eta_\eps>0$,
there exists $x_{\eps}$ such that
$$
  u^{\eps} (x) - \psi (x) \geq u^{\eps} (x_{\eps}) - \psi (x_{\eps})-\eta_\eps.
$$
Moreover, considering $\tilde{\psi}= \psi - u^{\eps} (x_{\eps}) -
\psi (x_{\eps})$, we can assume that $\psi (x_{\eps}) = u^{\eps} (x_{\eps})$.

If $u(x)<0$, we have to show that
\[
-\Delta_\infty \psi(x)\ge 0.
\]
Since $u$ is continuous and $u^\eps$ converges uniformly, we can assume that $u^\eps(x_\eps)<0$.
Thus, by recalling the fact that $u^\eps$ satisfy the DPP (Theorem \ref{teo:DPP}), and observing that
\[
\max \Bigg\{0;
 \sup_{ B_\eps (x)} u^\eps(y) - \eps \Bigg\}\geq 0
\]
we conclude that
\[
u^\eps(x) =
\frac12  \Bigg( \sup_{B_\eps (x)}u^\eps
+ \inf_{B_\eps (x)} u^\eps \Bigg).
\]
We obtain
\[
 \eta_\eps \geq-\psi (x_{\eps})+\frac12  \Bigg( \max_{\ol B_\eps (x_\eps)}\psi
+ \min_{\ol B_\eps (x_\eps)} \psi \Bigg).
\]
and thus, by \eqref{eq:asymp-inf-lap}, and choosing $\eta_\eps =
o(\eps^2)$, we have
$$
0\ge \eps^2\Big\langle
D^2\psi(x)\left(\frac{x_1^{\eps}-x}{\eps}\right),
\left(\frac{x_1^{\eps}-x}{\eps}\right)\Big\rangle +o(\eps^2).
$$
Next, we observe that
\[
\Big\langle
D^2\psi(x_{\eps})\left(\frac{x_1^{{\eps}}-x_{{\eps}}}{{\eps}}\right),
\left(\frac{x_1^{{\eps}}-x_{{\eps}}}{{\eps}}\right)\Big\rangle \to \Delta_{\infty}\psi(x)
\]
provided $\nabla \psi (x)\neq 0$. Furthermore, such a limit is bounded bellow and above by the quantities $\lambda_{\min} (D^2\psi(x))$ and $\lambda_{\max} (D^2\psi(x))$. Therefore, by dividing by $\eps^2$ and letting $\eps \to 0$, we get the desired inequality.

If $u(x)\geq 0$, we have to show that
\[
\max\{-\Delta_\infty \psi(x), 1-\abs{\nabla\psi(x)}\}\ge 0.
\]

As above, by \eqref{eq:asymp-inf-lap} and \eqref{eq:asymp-jensen}, we obtain
\begin{equation}
\begin{split}
0
&\ge \min\Bigg\{\frac{\eps^2}{2} D^2\phi(x)\left(\tfrac{x_1^{\eps}-x}{\eps}\right)\cdot
\left(\tfrac{x_1^{\eps}-x}{\eps}\right)+o(\eps^2);
\max\Bigg\{o(\eps^2)-\psi(x);\\
& \qquad \qquad \qquad \left(D \phi(x)\cdot
\tfrac{x_2^\eps-x}{\eps}-1\right)\eps
+\frac{\eps^2}{2}
D^2\phi(x)\left(\tfrac{x_2^{\eps}-x}{\eps}\right)\cdot
\left(\tfrac{x_2^{\eps}-x}{\eps}\right)+o(\eps^2)\Bigg\}\Bigg\}.
\end{split}
\end{equation}
and hence we conclude,
\[
\Delta_\infty \psi(x)\leq 0  \qquad\text{or}\qquad |\nabla \psi(x)|-1\leq 0
\]
as desired.

We have showed that $u$ is a supersolution to our equation. Similarly we obtain the subsolution counterpart. Let us remark, as part of those computations, that when $u^\eps(x)>0$ the DDP implies
\[
\max \Bigg\{0;
 \sup_{ B_\eps (x)} u^\eps(y) - \eps \Bigg\}> 0
\]
and hence
\[
 \sup_{ B_\eps (x)} u^\eps(y) - \eps > 0.
\]
Then in this case we have
\[
u^\eps(x) = \min \Bigg\{
\frac12  \left( \sup_{B_\eps (x)}u^\eps
+ \inf_{B_\eps (x)} u^\eps \right); \sup_{ B_\eps (x)} u^\eps(y) - \eps\Bigg\}.
\]
\end{proof}

We proved (see Theorem \ref{teo.unicidad}) that viscosity solutions to \eqref{eq:u>0} are unique
by using pure PDE methods. Therefore, we conclude that convergence as $\eps \to 0$ of $u^\eps$ holds not only along subsequences. This ends the proof of Theorem \ref{corofin}.

\section{Further properties for limit solutions}

Now, we present some relevant geometric and measure theoretic properties for limit solutions and their free boundaries.

 \begin{theorem}[{\bf Uniform positive density}]\label{UPDFB}Let $u_{\infty}$ be a limit solution to \eqref{MThmLim1} in $B_1$ and $x_0 \in \partial \{v > 0\} \cap B_{\frac{1}{2}}$ be a free boundary point. Then for any $0<\rho< \frac{1}{2}$,
$$
     \mathcal{L}^N(B_{\rho}(x_0) \cap\{u_{\infty}>0\})\geq \theta \rho^N,
$$
for a universal constant $\theta>0$.
\end{theorem}
\begin{proof} Applying Theorem \ref{MThmLim2} there exists a point $\hat{y} \in  \partial B_r(x_0) \cap \{u_{\infty}>0\}$ such that,
\begin{equation}\label{dens}
    v(\hat{y})\geq  r.
\end{equation}
Moreover, we claim that there exists a universal $\kappa>0$ small enough such that
\begin{equation}\label{inclusion}
    B_{\kappa r}(\hat{y}) \subset  \{u_{\infty}>0\},
\end{equation}
where the constant $\kappa$ is given by
$$
   \kappa \defeq \frac{1}{10[u_{\infty}]_{\text{Lip}(\Omega)}}.
$$
In fact, if this does not holds, it exists a free boundary point $\hat{z} \in B_{\kappa r}(\hat{y})$. Then, from \eqref{dens} we obtain
$$
   r \leq u_{\infty}(\hat{y}) \leq \sup_{B_{\kappa r}(\hat{z})} u_{\infty}(x) \leq [u_{\infty}]_{\text{Lip}(\Omega)}(\kappa r) = \frac{1}{10}r,
$$
which is a contradiction. Therefore,
$$
    B_{\kappa r}(\hat{y}) \cap B_r(x_0) \subset  B_r(x_0) \cap \{u_{\infty}>0\},
$$
and hence
$$
     \mathcal{L}^N(B_{\rho}(x_0) \cap\{u_{\infty}>0\})\geq \mathcal{L}^N(B_{\rho}(x_0) \cap B_{\kappa r}(\hat{y}))\geq \theta r^N,
$$
which proves the result.
\end{proof}

\begin{definition}[{\bf $\zeta$-Porous set}] A set $\mathfrak{S} \in \R^N$ is said to be porous with porosity constant $0<\zeta \leq 1$ if there exists an $\mathfrak{R} > 0$ such that for each $x \in \mathfrak{S}$ and $0 < r < \mathfrak{R}$ there exists a point $y$ such that $B_{\zeta r}(y) \subset B_r(x) \setminus \mathfrak{S}$.
\end{definition}

\begin{theorem}[{\bf Porosity of limiting free boundary}]\label{CorPor} Let $u_{\infty}$ be a limit solution to \eqref{MThmLim1} in $\Omega$. There exists a constant $0<\xi =  \xi(N, \text{Lip}[g]) \leq 1$ such that
\begin{equation}\label{eqPor}
       \mathcal{H}^{N-\xi}\left(\partial \{u_{\infty}>0\}\cap B_{\frac{1}{2}}\right)< \infty.
\end{equation}
\end{theorem}

\begin{proof}
Let $\mathfrak{R}>0$ and $x_0\in\Omega$ be such that $\overline{B_{4\mathfrak{R}}(x_0)}\subset \Omega$. We will show that $\partial \{u_{\infty} >0\} \cap B_\mathfrak{R}(x_0)$ is a $\frac{\zeta}{2}$-porous set for a universal constant $0< \zeta \leq 1$. To this end, let $x\in \partial \{u_{\infty} >0\} \cap B_{\mathfrak{R}}(x_0)$. For each $r\in(0, \mathfrak{R})$ we have $\overline{B_r(x)}\subset B_{2\mathfrak{R}}(x_0)\subset\Omega$. Now, let $y\in\partial B_r(x)$ such that $u_{\infty}(y) = \sup\limits_{\partial B_r(x)} u(t)$. From Theorem \ref{MThmLim2}
\begin{equation}\label{5.1}
    u_{\infty}(y)\geq r.
\end{equation}
On the other hand, near the free boundary, from Lipschitz regularity we have
\begin{equation}\label{5.2}
    u_{\infty}(y)\leq [u_{\infty}]_{\text{Lip}(\overline{\Omega})}.d(y),
\end{equation}
where $d(y) \defeq \dist(y, \partial \{u_{\infty}>0\} \cap \overline{B_{2\mathfrak{R}}(x_0)})$. From \eqref{5.1} and \eqref{5.2} we get
\begin{equation}\label{5.3}
    d(y)\geq\zeta r
\end{equation}
for a positive constant $0<\zeta \defeq \left(\frac{1}{[u_{\infty}]_{\text{Lip}(\overline{\Omega})}+1}\right)<1$.

Now, let $\hat{y}$, in the segment joining  $x$ and $y$, be such that $|y-\hat{y}|=\frac{\zeta r}{2}$, then there holds
\begin{equation}\label{5.4}
   B_{\frac{\zeta}{2}r}(\hat{y})\subset B_{\zeta r}(y)\cap B_r(x),
\end{equation}
indeed, for each $z\in B_{\frac{\zeta}{2}r}(\hat{y})$
\begin{align*}
   |z-y|&\leq |z-\hat{y}|+|y-\hat{y}|<\frac{\zeta r}{2}+\frac{\zeta r}{2}=\zeta r,\\
   |z-x|&\leq|z-\hat{y}|+\big(|x-y|-|\hat{y}-y|\big)\leq\frac{\zeta r}{2}+\left(r-\frac{\zeta r}{2}\right)=r.
\end{align*}
Then, since by \eqref{5.3} $B_{\zeta r}(y)\subset B_{d(y)}(y)\subset\{u_{\infty}>0\}$, we get $B_{\zeta r}(y)\cap B_r(x)\subset\{u_{\infty}>0\}$, which together with \eqref{5.4} implies that
$$
   B_{\frac{\zeta}{2}r}(\hat{y})\subset B_{\zeta r}(y)\cap B_r(x)\subset B_r(x)\setminus\partial\{u_{\infty}>0\}\subset B_r(x)\setminus \partial\{u_{\infty}>0\} \cap B_{\mathfrak{R}}(x_0).
$$
Therefore, $\partial\{v>0\} \cap B_{\mathfrak{R}}(x_0)$ is a $\frac{\zeta}{2}$-porous set. Finally, the $(N-\xi)$-Hausdorff measure estimates in \eqref{eqPor} follows from \cite{KR}.
\end{proof}

In particular, Corollary \ref{CorPor} implies that the free boundary $\partial \{u_{\infty}>0\}$ has Lebesgue measure zero.

\begin{theorem}[{\bf Convergence of the free boundaries}]\label{MThmLim5} Let $u_p$ be a sequence of solutions to \eqref{Eqp-Lapla}, $u_{\infty}$ its uniform limit and $x_0 \in \{u_{\infty}>0\}\cap \Omega^{\prime}$. Then
$$
  \partial \{u_p > 0\} \to \partial \{u_{\infty} > 0\}\quad \mbox{as} \quad  p\to \infty,
$$
locally in the sense of the Hausdorff distance.
\end{theorem}
\begin{proof}
Given  $\delta >0$ let $\mathcal{N}_{\delta}(\mathfrak{S}) \defeq \left\{ x \in \R^N : \dist(x, \mathfrak{S}) < \delta\right\}$ be the $\delta$-neighbourhood of a set $\mathfrak{S} \subset \R^N$. We must show that, given $0<\delta \ll 1$ and $p = p(\delta)$ large enough, one obtains
$$
  \partial \{v_p>0\} \subset \mathcal{N}_{\delta} (\partial\{v_{\infty}>0\}) \quad \mbox{and} \quad \partial\{v_{\infty}>0\} \subset \mathcal{N}_{\delta} (\partial \{v_p>0\}).
$$
We proceed with the first inclusion and suppose that it does not hold. Thus, it should exist a point $x_0 \in \partial \{u_p>0\} \cap \left(\Omega \setminus \mathcal{N}_{\delta} (\partial\{u_{\infty}>0\})\right)$. The last sentence implies  $\displaystyle \dist(x_0,  \partial\{u_{\infty}>0\}) \geq \delta$.

Now, if $x_0 \in \{u_{\infty} >0\}$ then by Corollary \ref{coro62} we get
$$
   u_{\infty}(x_0) \geq C \dist(x_0, \partial \{u_{\infty}>0 \}) \geq  C.\delta.
$$
On the other hand, due to the uniform convergence, for  $p$ large enough
$$
    u_p(x_0) \geq \frac{1}{100}  C\delta >0.
$$
However, this contradicts the assumption that $x_0 \in \partial \{u_p>0\}$. Therefore, $u_{\infty}(x_0) = 0$ and then $u_{\infty} \equiv 0$ in $B_{\delta}(x_0)$, which contradicts the strong non-degeneracy property given in Theorem \ref{LGR} since
$$
  \displaystyle \sup_{B_{\frac{\delta}{2}}(x_0)} u_p(x) \geq \mathfrak{c}\frac{\delta}{2}
$$
from where follows the inclusion.  The second inclusion can be proved similarly and we omit it.
\end{proof}

In  the following result, we analyse the behaviour of the coincidence sets for the $p$-variational problem and its corresponding limiting problem. We recall the following notion of limits  sets
$$
  \displaystyle \liminf_{p \to \infty} \,\mathrm{U}_p \defeq \bigcap_{p=1}^{\infty} \bigcup_{k \geq p} \mathrm{U}_k \quad \mbox{and} \quad \limsup_{p \to \infty} \,\mathrm{U}_p \defeq \bigcup_{p=1}^{\infty} \bigcap_{k \geq p} \mathrm{U}_k,
$$
and we say that there exists the limit
$
 \displaystyle \lim_{p \to \infty} \,\mathrm{U}_p$ when $ \displaystyle \liminf_{p \to \infty} \,\mathrm{U}_p =  \limsup_{p \to \infty} \,\mathrm{U}_p$.

\begin{theorem} Let $\mathrm{U}_p \defeq \{u_p = 0\}$ be the null sets of the problems \eqref{Eqp-Lapla} and $\mathrm{U}_{\infty} \defeq \{u_{\infty} = 0\}$ be the corresponding null set of the limiting problem.  Assume that along a subsequence $u_p \to u_\infty$. Then, also along the same subsequence, the null sets converge, that is,
$$
 \mathrm{U}_{\infty} = \lim_{p \to \infty} \,\mathrm{U}_p .
$$
\end{theorem}

\begin{proof} We will show that $\displaystyle
 \mathrm{U}_{\infty} \subset \liminf_{p \to \infty} \,\mathrm{U}_p \subset\limsup_{p \to \infty} \,\mathrm{U}_p \subset \mathrm{U}_{\infty}
$.
Given $0< \varepsilon \ll1$ (small enough), consider $\mathcal{V}_{\varepsilon}$ an $\varepsilon$-neighbourhood of $\mathrm{U}_{\infty}$. Thus, $\Omega \setminus \mathcal{V}_{\varepsilon} \subset \{u_{\infty}>0\}$ is a closed set. From the continuity of $u_{\infty}$ there exists a positive $\delta = \delta(\varepsilon)$ such that
$
   u_{\infty}(x)> \delta \quad \forall \,\, x \in \Omega \setminus \mathcal{V}_{\varepsilon}$. Moreover, by the uniform convergence (up a subsequence $u_p \to u_{\infty}$) we obtain that for $p$ large enough
$   u_{p}(x)> \delta \quad \forall \,\, x \in \Omega \setminus \mathcal{V}_{\varepsilon}$. Therefore $\Omega \setminus \mathcal{V}_{\varepsilon} \subset \{u_p>0\}$ from where $\mathrm{U}_p \subset \mathcal{V}_{\varepsilon}$ for every $p\gg 1$.
This implies that $ \limsup_{p \to \infty} \,\mathrm{U}_p \subset \mathcal{V}_{\varepsilon},$ for any $\varepsilon$-neighbourhood of $\mathrm{U}_{\infty}$. Hence, we obtain
$\displaystyle \limsup_{p \to \infty} \,\mathrm{U}_p \subset \mathrm{U}_{\infty}$
since $\mathrm{U}_{\infty}$ is a compact set.

Now, let $x_0 \in \mathrm{U}_{\infty}$. We claim that there exists a
sequence $x_p$ with $u_p(x_p) = 0$ such that $x_p\to x_0$.
In fact, from our previous estimates for $u_p$ we have near the free boundary
$$
  \mathfrak{c}_{\sharp} \dist(x, \partial \{u_p>0\})^{\frac{p}{p-1}} \leq u_p(x).
$$
Hence, in the set $\Omega_p =\{ x\in\Omega \, : \,  \dist(x, \partial \{u_p>0\}) \geq \delta \}$ we have that $u_p \geq C(\delta)$ (uniformly in $ p$). Now, as we have that $u_p (x_0) \to u_\infty (x_0) =0$, we obtain that $
\dist(x_0, \partial \{u_p>0\}) \to 0$  as  $p\to \infty$,
and we conclude that, given $\epsilon >0$, for every $p\geq p_0$ there is $x_p \in \mathrm{U}_p$ (that is, with $u_p(x_p) = 0$) such that $\dist(x_p, x_0) <\epsilon$. This shows that $\displaystyle  \mathrm{U}_\infty \subset \liminf_{p \to \infty} \,\mathrm{U}_p$,  as we wanted to prove.
\end{proof}

\begin{definition}[{\bf Reduced free boundary}] The \textit{reduced free boundary} $\mathfrak{F}^{\Omega}_{\text{red}}[u_{\infty}]$ is the set of points $x_0$ for whom it holds that, given the half ball $B_r^{+}(x_0) \defeq \{(x-x_0)\cdot\eta\geq 0\}\cap B_r(x_0)$ we get
\begin{equation}\label{eqRedFr}
    \displaystyle \lim_{r\to 0} \frac{\Leb(B_r^{+}(x_0) \triangle \Omega^{+}[u_{\infty}])}{\Leb(B_r(x_0))} = 0.
\end{equation}
This means (cf. \cite[Chapter 3]{Giusti}) that the vector measure $\nabla \chi_{\Omega}(B_r(x_0))$ has a density at the point, i.e., there exists $\eta(x_0)$ (with $|\eta(x_0)|=1$) such that fulfils the following
$$
  \displaystyle \lim_{r \to 0} \frac{\nabla \chi_{\Omega}(B_r(x_0))}{|\nabla \chi_{\Omega}(B_r(x_0))|} = \eta(x_0).
$$
\end{definition}

\begin{corollary} If $x_0 \in \mathfrak{F}^{\Omega}_{\text{red}}[u_{\infty}]$ then
$$
   B_r(x_0) \cap \mathfrak{F}^{\Omega}_{\text{red}}[u_{\infty}] \subset \{|(x-x_0)\cdot \eta(x_0)|\leq o(r)\} \quad \text{as} \quad r \to 0^+.
$$
\end{corollary}
\begin{proof}
  If we suppose that $u_{\infty}(x) =  0$ for $(x-x_0)\cdot\eta(x_0) \geq \varepsilon r$, then there exists $\mathfrak{c}_0 > 0$ such that $\Leb(B_{\varepsilon r}(x) \cap \{u_{\infty}=\}) \geq \mathfrak{c}_0\varepsilon r^N$, which implies, according to Corollary \ref{UPDFB}, that
$$
  \displaystyle \liminf_{r \to 0} \frac{\Leb(B_r^{+}(x_0) \triangle \Omega^{+}[u_{\infty}])}{\Leb(B_r(x_0))} \geq \mathfrak{c}_0\varepsilon,
$$
which contradicts \eqref{eqRedFr}.
\end{proof}

We finish this section proving that free boundary points having a tangent ball from inside are regular. To this end, let us introduce the following definition.

\begin{definition}[{\bf Regular points}] A free boundary point $y \in \mathfrak{F}_{\Omega}[u] \defeq \partial \{u>0\}\cap \Omega$ is said to have a \textit{tangent ball from inside} if there exists a ball $\mathcal{B} \subset \Omega^{+}[u] \defeq \{u>0\}\cap \Omega$ such that $y \in  \mathcal{B} \cap \Omega^{+}[u]$. Finally, a free boundary point $y \in \mathfrak{F}_{\Omega}[u]$ is \emph{regular} if $\mathfrak{F}_{\Omega}[u]$ has a tangent hyperplane at $y$.
\end{definition}

\begin{theorem}\label{RegFB} A free boundary point $y \in \mathfrak{F}_{\Omega}[u_{\infty}]$ for a limit solution of the problem \eqref{EqLim} which has a tangent ball from inside is regular.
\end{theorem}

\begin{proof}
The proof follows similarly to \cite[Lemma 11.17]{CafSal}, thus we only sketch the modifications for the reader's convenience. Let us suppose that $B_1(y_1)$ is tangent to $\mathfrak{F}_{\Omega}[u_{\infty}]$ at $y$ and consider the  function
$$
  \Phi(x) = 1-|x-y_1|.
$$
From the non-degeneracy, some multiple of $\Phi$, say $\mathfrak{c} \Phi$, is a lower barrier of
$u_{\infty}$ in $B_1(y_1)$. Now, let $\mathfrak{c}_r > 0$ be the supremum of all $\mathfrak{c}$'s such that $u(x) \geq \mathfrak{c} \Phi(x)$ in $B_r(y_1)$.

Notice that such values $\mathfrak{c}_r$ increase with $r$. Hence, by optimal regularity,
$\mathfrak{c}_r$ converges to some constant $\mathfrak{c}_{\infty}$ as $r \to 0$.
According to \cite[Lemma 11.17]{CafSal}, this implies the following asymptotic behaviour near the free boundary
$$
  u_{\infty}(x) =  \mathfrak{c}_{\infty} (x-y)\cdot\eta(y) + o \left((x-y)\cdot\eta(y)\right),
$$
where $\eta(y) = y_1-y$. Therefore, the plane orthogonal to $\eta(y)$ is tangent to $\mathfrak{F}_{\Omega}[u_{\infty}]$ and, we conclude that $y$ is a regular point.
\end{proof}

In particular, the previous result reveals that at interior free boundary points verifying the interior ball  condition, limit solutions for the $p-$obstacle problem with zero constraint, are regular.

In the last part of this paper we include two examples to see what kind of solutions to \eqref{EqLim} one can expect.

\begin{example}[{\bf\bf Radial solutions}] First of all, let us study the following boundary value problem:
\begin{equation}\label{rad eq}
	\left \{
		\begin{array}{rllll}
			-\Delta_p u &=& -\lambda_0 \chi_{\{u>0\}}(x) & \text{ in } & B_{R}(x_0) \\
			u(x) &=& \kappa &\text{ on } & \partial B_{R}(x_0),
		\end{array}
	\right.
\end{equation}
where $R, \lambda_0$ and $\kappa$ are a positive constants.

Observe that by the uniqueness of solutions for the Dirichlet problem \eqref{rad eq} and invariance under rotations of the $p-$Laplace operator, it is easy to see that $u$ must be a radially symmetric function. Hence, let us deal with the following one-dimensional ODE
\begin{equation}\label{rad edo}
			-(|v^{\prime}(t)|^{p-2}v^{\prime}(t))^{\prime} = -\lambda_0 v_{+}(t) \quad \text{ in }\, (0, \mathfrak{T}), \qquad v(0)=0 \,\,\,\text{and}\,\,\,v(\mathfrak{T})=\kappa.
\end{equation}
It is straightforward to check that $v(t)=\Theta(1, \lambda_0, p) t^{\,\frac{p}{p-1}}$ is a solution to \eqref{rad edo}, where
\begin{equation}\label{theta}
	\Theta = \Theta(N, \lambda_0, p) \defeq  \frac{p-1}{p}\left(\frac{ \inf_{\Omega}\lambda_0(x)}{ N}\right)^{\frac{1}{p-1}}
	\quad \mbox{and} \quad
	\mathfrak{T} \defeq \left( \frac{\kappa}{\Theta} \right)^{\frac{p-1}{p}}. 	
\end{equation}
Now, in order to characterize the unique solution \eqref{rad eq} fix $x_0 \in \mathbb{R}^N$ and $0<\mathfrak{r}_0<R$. We assume the \textit{compatibility condition} for the dead-core problem, namely $R > \mathfrak{T}$. Thus, for $\mathfrak{r}_0 = R-\mathfrak{T}$ the radially symmetric function given by
\begin{equation}\label{RadProf}
  u(x)\defeq \Theta \left[|x-x_0|-R + \left( \frac{1}{\Theta} \right)^{\frac{p-1}{p}} \right]_+^{\frac{p}{p-1}} = \Theta
  \left(|x-x_0|-\mathfrak{r}_0 \right)_+^{\frac{p}{p-1}}
\end{equation}
fulfils \eqref{rad eq} in the weak sense, where $\mathfrak{r}_0 \defeq R - \left( \frac{\kappa}{\Theta} \right)^{\frac{p-1}{p}}$. Moreover, the dead-core is given  by $B_{\mathfrak{r}_0}(x_0)$.

Also it is easy to see that the limit radial profile as $p \to \infty$ becomes
\begin{equation} \label{radial}
  u_{\infty}(x)\defeq \left(|x-x_0|-\mathfrak{r}_0 \right)_+,
\end{equation}
which satisfies \eqref{EqLim} in the viscosity sense with $\Omega = B_{R}(x_0)$ the dead-core given by $B_{\mathfrak{r}_0}(x_0)$ for $\mathfrak{r}_0 = R - 1$ and $g\equiv \kappa$ on $\partial B_{R}(x_0)$.
\end{example}

\begin{example}Finally, by considering the one dimensional problem
$$
\left\{
\begin{array}{rclcl}
  \max\{- u'', \chi_{\{u>0\}}-|u'|\} & = & 0 &\text{ in } & (-1,4) \\
  u(-1) & = & 1& & \\
  u(4) & = & -1, & &
\end{array}
\right.
$$
it is straightforward to verify that $u(x) = \left\{
\begin{array}{rcl}
 -x  & \text{if} & x\in (-1, 0] \\
 -\frac{1}{4}x & \text{if} & x\in [0, 4)
\end{array}
\right.$
is the unique viscosity solution to our gradient constraint problem.
\end{example}

\subsection*{Acknowledgements}
This work was partially supported by Consejo Nacional de Investigaciones Cient\'{i}ficas y T\'{e}cnicas (CONICET-Argentina).

\end{document}